\documentclass[11pt]{amsart}
\usepackage{latexsym,amscd,amssymb, graphicx, color, amsthm, bm, amsmath, cancel, enumitem}  
\usepackage{tikz}
\usepackage[margin=1in]{geometry}
\usepackage{hyperref}
\usepackage[all,cmtip]{xy}
\usepackage{ytableau}

\numberwithin{equation}{section}

\newtheorem{theorem}{Theorem}[section]
\newtheorem{proposition}[theorem]{Proposition}
\newtheorem{corollary}[theorem]{Corollary}
\newtheorem{lemma}[theorem]{Lemma}
\newtheorem{conjecture}[theorem]{Conjecture}

\newtheorem{question}[theorem]{Question}

\newtheorem{remark}[theorem]{Remark}

\makeatletter
\newtheorem*{rep@theorem}{\rep@title}
\newcommand{\newreptheorem}[2]{%
\newenvironment{rep#1}[1]{%
 \def\rep@title{#2 \ref{##1}}%
 \begin{rep@theorem}}%
 {\end{rep@theorem}}}
\makeatother
\theoremstyle{definition}
\newtheorem{defn}[theorem]{Definition}

\newreptheorem{lemma}{Lemma}
\newreptheorem{theorem}{Theorem}

\def\cocoa{{\hbox{\rm C\kern-.13em o\kern-.07em C\kern-.13em o\kern-.15em A}}}

\newcommand{\sign}{{\mathrm {sign}}}

\newcommand{\Hilb}{{\mathrm {Hilb}}}

\newcommand{\symm}{{\mathfrak{S}}}

\newcommand{\CC}{{\mathbb {C}}}
\newcommand{\RR}{{\mathbb {R}}}

\newcommand{\ZZ}{{\mathbb {Z}}}

\newcommand{\OOO}{{\mathcal{O}}}
\newcommand{\KK}{{\mathbb{K}}}
\newcommand{\AAA}{{\mathcal{A}}}
\newcommand{\BBB}{{\mathcal{B}}}

\newcommand{\MMM}{{\mathcal{M}}}
\newcommand{\III}{{\mathcal{I}}}

\newcommand{\LL}{{\mathfrak{L}}}

\newcommand{\OS}{{\mathrm{OS}}}

\newcommand{\ee}{{\mathbf {e}}}

\newcommand{\g}{{\mathfrak {g}}}
\newcommand{\bb}{{\mathfrak {b}}}
\newcommand{\ttt}{{\mathfrak {t}}}
\newcommand{\h}{{\mathfrak {h}}}

\newcommand{\ii}{{\mathfrak {i}}}
\newcommand{\aaa}{{\mathfrak {a}}}
\newcommand{\ccc}{{\mathfrak {c}}}

\newcommand{\ST}{{\mathcal {ST}}}

\newcommand{\stair}{{\mathrm{st}}}
\newcommand{\Der}{{\mathrm {Der}}}
\newcommand{\Hom}{{\mathrm {Hom}}}

\begin{document}

\title[Superspace coinvariants and hyperplane arrangements]
{Superspace coinvariants and hyperplane arrangements}

\author[R. Angarone]{Robert Angarone}
\author[P. Commins]{Patricia Commins}
\author[T. Karn]{Trevor Karn}
\author[S. Murai]{Satoshi Murai}
\author[B. Rhoades]{Brendon Rhoades}

\address
{School of Mathematics \newline \indent
University of Minnesota \newline \indent
Minneapolis, MN, 55455, USA}
\email{(angar017, commi010, karnx018)@umn.edu}

\address
{Department of Mathematics \newline \indent
Faculty of Education \newline \indent
Waseda University \newline \indent
1-6-1 Nishi-Waseda, Shinjuku, Tokyo, 169-8050, Japan}
\email{s-murai@waseda.jp}

\address
{Department of Mathematics \newline \indent
University of California, San Diego \newline \indent
La Jolla, CA, 92093, USA}
\email{bprhoades@ucsd.edu}

\begin{abstract}
Let $\Omega$ be the {\em superspace ring} of polynomial-valued differential forms on affine $n$-space.
The natural action of the symmetric group $\symm_n$ on $n$-space induces an action of $\symm_n$ on $\Omega$.
The {\em superspace coinvariant ring} is the quotient $SR$ of $\Omega$ by the ideal generated by $\symm_n$-invariants with vanishing constant term.
We give the first explicit basis of $SR$, proving a conjecture of Sagan and Swanson. 
Our techniques use the theory of hyperplane arrangements. We relate $SR$ to instances of the 
Solomon--Terao algebras of Abe--Maeno--Murai--Numata and use exact sequences relating the derivation modules of 
certain `southwest closed' arrangements to obtain the desired basis of $SR$.
\end{abstract}

\maketitle

\section{Introduction}
\label{Introduction}

Let $\KK$ be a field of characteristic 0 and let $n$ be a positive integer.
We let $(x_1, \dots, x_n)$ be a list of $n$ commuting variables and write $S := \KK[x_1, \dots, x_n]$
for the polynomial ring over these variables.
The {\em superspace} ring of rank $n$ is the tensor product
\begin{equation}
\Omega := \KK[x_1, \dots, x_n] \otimes \wedge \{ \theta_1, \dots, \theta_n \}
\end{equation}
where 
$(\theta_1, \dots, \theta_n)$ is a list of $n$ anticommuting variables.
The ring  $\Omega$ is a bigraded $\KK$-algebra
with $\Omega_{i,j} = \KK[x_1, \dots, x_n]_i \otimes \wedge^j \{ \theta_1, \dots, \theta_n \}$.
Borrowing terminology from physics, 
we refer to the $x$-variables as {\em bosonic}  and the $\theta$-variables as {\em fermionic}.
Regarding $\Omega$ as an algebra of polynomial-valued
differential forms on affine $n$-space, we have the {\em Euler operator} or {\em total derivative}
$d: \Omega \rightarrow \Omega$ given by
\begin{equation}
d f := \sum_{i \, = \, 1}^n \partial_i(f) \cdot \theta_i,
\end{equation}
where the partial derivative $\partial_i = \partial/\partial x_i$ treats $\theta$-variables as constants.
The symmetric group $\symm_n$ acts `diagonally' on $\Omega$ by subscript permutation, viz.
\begin{equation}
w \cdot x_i := x_{w(i)} \quad \quad 
w \cdot \theta_i := \theta_{w(i)} \quad \quad 
(w \in \symm_n, \, \, 1 \leq i \leq n).
\end{equation}

Quotients of the ring $\Omega$ have received increasing attention in recent years
with ties to  Chern plethysm \cite{BRT},
 Tutte polynomials of vector configurations \cite{RTW},
 and delta conjecture modules \cite{RWVan, RWSuper}.
A {\em differential ideal} is a (two-sided) ideal in $\Omega$ which is closed under the Euler operator $d$.
One way to construct a bigraded 
differential ideal in $\Omega$ is to take the {\em differential closure} of a singly graded ideal $J \subseteq S$: the smallest differential
ideal in $\Omega$ which contains $J$.
Consider the type A {\em coinvariant ideal} $(S^{\symm_n}_+) \subset S$ generated by $\symm_n$-invariants with vanishing constant term.
We write $I \subset \Omega$ for the differential closure of $(S^{\symm_n}_+)$.  
Solomon proved \cite{Solomon} that $I$ admits a straightforward generating set: it is generated by $\symm_n$-invariants in $\Omega$ with vanishing constant term.
The {\em superspace coinvariant ring} is the quotient
\begin{equation}
SR := \Omega/I
\end{equation}
of $\Omega$ by $I$.  The ring $SR$ is a bigraded $\symm_n$-module.

The superspace coinvariant ring $SR$ has been the subject of a significant amount of recent research.
In 2018, the Fields Institute Combinatorics Group\footnote{Nantel Bergeron, Laura Colmenarajo, Shu Xiao Li, John Machacek, Robin Sulzgruber, and Mike Zabrocki}
made a family of conjectures on the algebraic structure of $SR$ which predicted its bigraded $\symm_n$-character; see \cite{Zabrocki}.
Swanson--Wallach \cite{SW1,SW2} gave a conjectural description of the Macaulay inverse system of $SR$ and
Sagan--Swanson \cite{SS} gave a conjectural basis of $SR$.
Despite superficial similarity between $SR$ and other algebraic objects \cite{HRS, RWVan, RWSuper} which exhibit analogous algebraic behavior,
these conjectures resisted proof. 
A large source of difficulty was the inscrutable Gr\"obner theory of 
the ideal $I \subset \Omega$, making the quotient $SR = \Omega/I$ resistant to `straightening' arguments.


In 2023, Rhoades and Wilson \cite{RW} used algebraic techniques to calculate the bigraded Hilbert series
of $SR$, confirming a conjecture of the Fields Group, and proved the Operator Conjecture of Swanson--Wallach \cite{SW1,SW2}
on the Macaulay inverse system $I^\perp \subseteq \Omega$.
The Hilbert series of $SR$ is consistent with a conjectural monomial basis of $SR$ predicted by Sagan and Swanson.
In order to state the Sagan--Swanson conjecture, we need some notation.

For any subset $J \subseteq [n] := \{1, \dots, n \}$,  the
{\em $J$-staircase} $(\stair(J)_1, \dots, \stair(J)_n)$ is defined recursively by
\begin{equation}\label{eqn:staircase}
\stair(J)_1 = \begin{cases}
1 & 1 \notin J, \\
0 & 1 \in J,
\end{cases}  \quad \quad
\stair(J)_i = \begin{cases}
\stair(J)_{i-1} + 1 & i \notin J, \\
\stair(J)_{i-1} & i \in J.
\end{cases}
\end{equation}
For example, if $n = 5$ and $J = \{2,4\}$ we have $(\stair(J)_1, \dots, \stair(J)_5) = (1,1,2,2,3)$.
We let $\MMM(J)$ be the family of bosonic monomials
\begin{equation}
\MMM(J) := \{ x_1^{a_1} \cdots x_n^{a_n} \,:\, a_i < \stair(J)_i \}
\end{equation}
which fit under the $J$-staircase.  In our example $n = 5$ and $J = \{2,4\}$ we have
\begin{equation*}
\MMM(J) = \{ x_3 x_4 x_5^2, \; x_3 x_4 x_5,\;  x_3 x_4,\;  x_3 x_5^2,\;  x_3 x_5,\;  x_3,\;  x_4 x_5^2,\;  x_4 x_5,\;  x_4,\;  x_5^2,\;  x_5,\;  1 \}.
\end{equation*}
Observe that $\MMM(J) = \varnothing$ when $1 \in J$.
Given a subset $J \subseteq [n]$, let $\theta_J$ be the corresponding product of $\theta$-variables with increasing subscripts.
The {\em superspace Artin monomials} are given by
\begin{equation}
\MMM := \bigsqcup_{J \, \subseteq \, [n]} \MMM(J) \cdot \theta_J.
\end{equation}
When $n = 3$ we have
\begin{equation*}
\MMM = \{ x_2 x_3^2, x_2 x_3, x_2, x_3^2, x_3, 1 \} \sqcup \{ x_2 x_3 \theta_3, x_2 \theta_3, x_3 \theta_3, \theta_3 \} \sqcup \{ x_3 \theta_2, \theta_2 \} \sqcup \{ \theta_2 \theta_3 \}.
\end{equation*}

\begin{conjecture}
\label{conj:ss}
{\em (Sagan--Swanson \cite{SS})}
The set $\MMM$ descends to a basis of $SR$.
\end{conjecture}

We will prove Conjecture~\ref{conj:ss}, giving the first explicit basis of the quotient $SR$; see Corollary~\ref{cor:ss}.
Our starting point is the following transfer principle of Rhoades and Wilson \cite{RW}
which allows us to trade a single problem in supercommutative algebra for a family of problems in commutative algebra.
Given $J \subseteq [n]$, let $f_J \in S$ be the polynomial
\begin{equation}
f_J := \prod_{j \, \in \, J} x_j \times \left( \prod_{i \, = \, j+1}^n (x_j - x_i) \right).
\end{equation}
Recall that if $\aaa \subseteq S$ is an ideal and $f \in I$, the {\em colon ideal} (or {\em ideal quotient}) $\aaa : f \subseteq S$ is given by
\begin{equation}
\aaa : f = \{ g \in S \,:\, fg \in \aaa \}.
\end{equation}
It is not hard to see that $\aaa : f$ is an ideal in $S$ containing $\aaa$.
Rhoades and Wilson proved \cite{RW} that the colon ideals $(S^{\symm_n}_+) : f_J$ exhibit the following dichotomy:
\begin{enumerate}
\item if $1 \in J$ then $(S^{\symm_n}_+) : f_J$ is the unit ideal $S$, and
\item if $1 \notin J$ then $(S^{\symm_n}_+) : f_J$ is generated by a homogeneous regular sequence 
$p_{J,1}, \dots, p_{J,n} \in S$ with $\deg(p_{J,i}) = \stair(J)_i$.
\end{enumerate}
The polynomials $p_{J,i}$ appearing in \cite{RW} have a complicated definition involving partial derivatives.

\begin{theorem}
\label{thm:rw}
{\em (Rhoades--Wilson \cite{RW})}
For each $J \subseteq [n]$, suppose that $\BBB(J) \subseteq S$ is a set of homogeneous polynomials which descends to a basis
of $S/( (S^{\symm_n}_+) : f_J)$. Then 
\begin{equation*}
\BBB := \bigsqcup_{J \, \subseteq \, [n]} \BBB(J) \cdot \theta_J
\end{equation*}
descends to a basis of $SR$.
\end{theorem}

Theorem~\ref{thm:rw} was crucial for calculating the Hilbert series of $SR$ and gives a recipe for finding bases of $SR$.
The Sagan--Swanson Conjecture~\ref{conj:ss} would follow if we could show that $\MMM(J)$ descends to a basis of the 
commutative quotient $S/( (S^{\symm_n}_+) : f_J)$.
However, prior to this paper, this recipe resisted execution. 
Although the set $\MMM(J)$ was shown \cite{RW} to descend to a basis of $S/( (S^{\symm_n}_+) : f_J)$ when $J = \{r,r+1, \dots, n\}$ is a `terminal' subset of $[n]$,
the mysterious Gr\"obner theory of the ideals $(S^{\symm_n}_+) : f_J$ obstructed a proof for general $J$.
In this paper we will prove (Theorem~\ref{thm:ss-J})
that $\MMM(J)$ descends to a basis for all $J$ by developing a surprising connection to the theory of
hyperplane arrangements.

Let $V = \KK^n$ be the standard $n$-dimensional vector space over $\KK$.
In this paper, a {\em hyperplane} will be a linear codimension 1 subspace $H \subset V$ and a {\em hyperplane arrangement}
will be a finite set $\AAA$ of hyperplanes. If $\alpha_H \in S$ is a linear form with $\mathrm{Ker}(\alpha_H) = H$, we set $Q(\AAA) := \prod_{H \in \AAA} \alpha_H$. The polynomial $Q(\AAA) \in S$ is defined up to a nonzero scalar.
Two arrangements which will be important to us are the 
{\em braid arrangement} $\AAA_{\Phi^+}$ given by the kernels of the linear forms
\begin{equation}
\Phi^+ := \{ x_i - x_j \,:\, 1 \leq i < j \leq n \}
\end{equation}
and the {\em augmented braid arrangement} $\AAA_{\widetilde{\Phi}^+}$ corresponding to 
\begin{equation}
\widetilde{\Phi}^+ := \{ x_i - x_j \,:\, 1 \leq i < j \leq n \} \cup \{ x_1, \dots, x_n \}
\end{equation}
where $x_1, \dots, x_n \in V^*$ are the usual coordinate functions.

The derivation module $\Der(S)$ of the polynomial ring $S$ is
the free $S$-module of rank $n$ given by
\begin{equation}
\Der(S) = \bigoplus_{i \, = \, 1}^n S \cdot \partial_i
\end{equation}
where $\partial_i: S \rightarrow S$ acts by differentiation with respect to $x_i$. Let $I_\AAA = Q(\AAA) S$ be the ideal of polynomials which vanish on $\AAA$. A derivation $\rho \in \Der(S)$ is a {\em derivation of $\AAA$} if $\rho(I_\AAA) \subseteq I_\AAA$, or equivalently if $Q(\AAA) \mid \rho(Q(\AAA))$. Let $\Der(\AAA)$ denote the $S$-submodule of $\Der(S)$ consisting of all derivations of $\AAA$. The arrangement $\AAA$ is \emph{free} if $\Der(\AAA)$ is a free $S$-module. For example, the arrangement $\AAA_{\Phi^+}$ is free with basis $\{ \sum_{i = 1}^n x_i^k \partial_i \,:\, 0 \leq k < n \}$; see  \cite[Ex. 4.32]{OT}.

Sending $\partial_i \mapsto c_i$ for some fixed $c_1, \ldots, c_n \in S$ determines an $S$-module map $\ccc \colon \Der(S) \to S$. We call $\ccc_\AAA := \ccc(\Der(\AAA)) \subset S$ a \emph{Solomon--Terao ideal}. For example, when $\aaa: \partial_i \mapsto x_i$, the  Solomon--Terao ideal $\aaa_{\AAA_{\Phi^+}} = (\sum_{i=1}^n x_i^k : 1 \leq k \leq n)$ corresponding to $\AAA_{\Phi^+}$ is the classical coinvariant ideal. We will assume throughout that all $c_i$ are homogeneous of the same fixed degree $d \geq 0$.

\begin{defn}
\label{def:st-algebra}
Given $\AAA$ and $\ccc$ as above, the {\em Solomon--Terao algebra} is the quotient ring
\begin{equation}
\ST(\AAA;\ccc) := S / \ccc_\AAA.
\end{equation}
\end{defn}

The homogeneity assumption on $\ccc$ implies that $\ST(\AAA;\ccc)$ is a graded quotient of $S$.
Definition~\ref{def:st-algebra} is slightly more general than the Solomon--Terao algebras defined 
by Abe, Maeno, Murai, and Numata \cite{AMMN}; see Section~\ref{ST}.
When the arrangement $\AAA$ is free and $\ccc$ satisfies certain technical conditions, the Solomon--Terao algebras provide a vast 
array of quotients of $S$ cut out by complete intersections; see Lemma~\ref{lem:st-hilbert-series}.
Furthermore, short exact sequences involving derivation modules of certain triples of arrangements 
induce short exact sequences 
of the corresponding Solomon--Terao algebras (Lemma~\ref{lem:st-short-exact}).

We introduce a family of {\em southwest arrangements} $\AAA$ as subarrangements of $\AAA_{\widetilde{\Phi}^+}$.
The southwest arrangements are shown to be free (Proposition~\ref{prop:sw-free}). 
If we let $\ii: \Der(S) \rightarrow S$ be the $S$-module map induced by 
$\ii: \partial_i \mapsto 1$ for all $i$, we show (Proposition~\ref{prop:st-general})
 that for appropriate southwest arrangements $\AAA$ the Solomon--Terao algebra
$\ST(\AAA;\ii)$ coincides with the quotient $S/((S^{\symm_n}_+) : \widetilde{f}_J)$ where
$\widetilde{f}_J \in S$ has the same definition as $f_J$, but without the $x_j$ factors.
Southwest arrangements admit a recursive structure under deletion and restriction (Lemma~\ref{lem:sw-recursion}) which 
gives rise to short exact sequences leading to an inductive proof (Theorem~\ref{thm:ss-J}) that 
$\MMM(J)$ descends to a basis of $S/((S^{\symm_n}_+) : f_J)$ for all subsets  $J \subseteq [n]$,
thereby proving (Corollary~\ref{cor:ss}) the Sagan--Swanson Conjecture~\ref{conj:ss}.
We  use similar methods and a family of arrangements $\AAA_J$ distinct from southwest arrangements to
obtain (Corollary~\ref{cor:generating-set}) a simpler regular sequence which cuts out the ideal
$(S^{\symm_n}_+) : f_J$ appearing in Theorem~\ref{thm:rw}.

The methods in this paper have parallels in Hessenberg theory.
Let $G$ be a simple complex Lie group with Borel subgroup $B$ 
and maximal torus $T$.
Let $\g \supseteq \bb \supseteq \ttt$ be the corresponding Lie algebras.
If $\h \subseteq \g$ is a Hessenberg subspace,
we have a corresponding {\em regular nilpotent Hessenberg variety} $X_{\h}$ which is a closed subvariety of the flag variety $G/B$.

If $H^*(X_{\h})$ is the singular cohomology ring of $X_{\h}$ with real coefficients,
there is a surjective {\em Borel map} $\RR[\ttt] \twoheadrightarrow H^*(X_{\h})$ which realizes 
$H^*(X_{\h})$ as a quotient of the polynomial ring $\RR[\ttt]$.
The Hessenberg subspace $\h$ also corresponds to an ideal $\III$ in the positive root poset
$\Phi^+$ attached to $G$. 
Abe, Horiguchi, Masuda, Murai, and Sato proved \cite{AHMMS} (using different language) that 
$H^*(X_{\h})$ is the Solomon--Terao algebra associated to the ideal subarrangement $\AAA_\III$ of the Weyl arrangement
with respect to the map $\aaa: \partial_i \mapsto x_i$.
This uniform presentation of $H^*(X_{\h})$ allowed Abe et. al. to deduce numerous properties of these rings,
including the K\"ahler package (despite the failure of the $X_{\h}$ to be smooth).
Harada, Horiguchi, Murai, Precup, and Tymoczko \cite{HHMPT} showed that $H^*(X_{\h})$ admits an Artin-like basis in type A.

The Solomon--Terao algebras arise surprisingly in two seemingly unrelated contexts: Hessenberg varieties and superspace coinvariant theory.
In Section~\ref{Conclusion} we discuss the problem of finding other applications of these algebras.

The rest of the paper is organized as follows.
In {\bf Section~\ref{Background}} we give general background on commutative algebra and hyperplane arrangements.
In {\bf Section~\ref{ST}} we study general properties of Solomon--Terao algebras and prove a result
(Lemma~\ref{lem:derivation_colon}) which 
relates the ideals $\ccc_\BBB, \ccc_\AAA \subseteq S$ for certain subarrangements $\BBB \subseteq \AAA$. 
{\bf Section~\ref{Southwest}} introduces southwest arrangements, finds bases for their derivation modules, 
and derives a short exact sequence for the derivation modules of triples of arrangements arising in Terao's Addition-Deletion Theorem.
{\bf Section~\ref{AJ}} defines and studies the arrangements $\AAA_J$. We use the derivation modules of these arrangements to give a simpler regular sequence generating the ideals $(S^{\symm_n}_+) : f_J$.
In {\bf Section~\ref{Monomial}} we derive monomial bases for Solomon--Terao algebras corresponding to southwest arrangements.
{\bf Section~\ref{Artin}} applies the theory of southwest arrangements to prove the Sagan--Swanson Conjecture.
{\bf Section~\ref{Conclusion}} gives some open problems.

\section{Background}
\label{Background}

\subsection{Commutative algebra}\label{sec:commutative-algebra}
As explained in the introduction, we let $S = \KK[x_1, \dots, x_n]$ where $\KK$ is a field of characteristic 0. If $V = \KK^n$, we will sometimes identify $S = \KK[V]$
with the ring of regular maps $V \rightarrow \KK$.

For any $1 \leq i \leq n$ we have a partial derivative operator $\partial_i: S \rightarrow S$. These operators commute, so for any 
polynomial $f = f(x_1, \dots, x_n) \in S$ we have a differential operator $\partial f: S \rightarrow S$ given by
\begin{equation}
\partial f := f(\partial_1, \dots, \partial_n).
\end{equation}
This gives rise to an $S$-module structure $(-) \odot (-): S \times S \rightarrow S$ given by
\begin{equation}
f \odot g := \partial f(g).
\end{equation}

Let $A = \bigoplus_{d \geq 0} A_d$ be a graded $\KK$-vector space in which each graded piece $A_d$ is finite-dimensional.
The {\em Hilbert series} of $A$ is the formal power series
\begin{equation}
\Hilb(A;q) := \sum_{d \, \geq \, 0} \dim_{\KK} A_d \cdot q^d.
\end{equation}

Let $\aaa \subseteq S$ be an ideal and let $f \in S$.  As in the introduction, the {\em colon ideal} 
$\aaa : f$ is the ideal in $S$ given by $\{ g \in S \,:\, f g \in \aaa \}$.
It is not hard to see that $\aaa : f$ is an ideal containing $\aaa$.  Furthermore, we have an exact sequence
\begin{equation}
0 \longrightarrow S/(\aaa : f) \xrightarrow{ \, \, \times f \, \, } S/\aaa
\end{equation}
induced by multiplication by $f$. If $f, g \in S$ we have $\aaa : fg = (\aaa : f) : g$ and a corresponding exact sequence
\begin{equation}
0 \longrightarrow S/(\aaa : fg) \xrightarrow{ \, \, \times g \, \, } S/(\aaa:f).
\end{equation}

A sequence $f_1, \dots, f_r \in S$ of polynomials is called {\em regular} if 
$(f_1, \dots, f_r) \neq S$ and for each $i$, 
the map
\begin{equation}
S/(f_1, \dots, f_{i-1}) \xrightarrow{ \, \, \times f_i \, \, } S/(f_1, \dots, f_{i-1})
\end{equation}
is an injection. If each $f_i$ is homogeneous, we have $r \leq n$ and
\begin{equation}
\label{eq:2.7}
\Hilb(S/(f_1, \dots, f_r); q) = \frac{(1-q^{\deg f_1}) \cdots (1- q^{\deg f_r})}{(1-q)^n}.
\end{equation}
An ideal $\aaa \subseteq S$ is a {\em complete intersection} if $\aaa$ is generated by a regular sequence.
Let $\aaa \subseteq S$ be a homogeneous ideal. 
If $f_1, \dots, f_n \in S$ have positive degree, then
$f_1, \dots, f_n$ is a regular sequence if and only if $S/\aaa$ is a finite-dimensional $\KK$-vector space.

Let $A = \bigoplus_{d = 0}^m A_d$ be a finite-dimensional graded $\KK$-algebra with $A_m \neq 0$.
The algebra $A$ is a {\em Poincar\'e duality algebra (PDA)} if $A_m \cong \KK$ is 1-dimensional and if, for all $d$, the 
multiplication map $A_d \times A_{m-d} \rightarrow A_m \cong \KK$ given by $(a, b) \mapsto ab$ is a {\em perfect pairing}. (The map $(a, b) \mapsto ab$ is a perfect pairing if $(a,b) \mapsto 0$ for all $b \in A_{m - d}$ implies $a = 0$ \textit{and} $(a, b) \mapsto 0$ for all $a \in A_d$ implies $b = 0.$)  
The integer $m > 0$ is the {\em socle degree} of $A$.
If a homogeneous ideal $\aaa \subseteq S$ is a complete intersection generated by a regular sequence $f_1, \dots, f_n$ of length $n$,
 it is known that $S/\aaa$ is a PDA
of socle degree $\sum_{i = 1}^n (\deg f_i - 1)$; see \cite[Thm.\ 6.5.1]{Smith}.
The following lemma will be crucial.

\begin{lemma}
\label{lem:colon-ideal-generation}
{\em (Abe--Horiguchi--Masuda--Murai--Sato \cite[Lem.\ 2.4]{AHMMS})}
Suppose $\aaa, \aaa' \subseteq S$ are homogeneous ideals and $f \in S$ is a homogeneous polynomial of degree $k$ with 
$f \notin \aaa$. Suppose $\aaa' \subseteq (\aaa : f)$. If $S/\aaa'$ is a PDA of socle degree $m$ and $S/\aaa$ is a PDA of socle 
degree $m+k$ then $\aaa' = (\aaa : f)$.
\end{lemma}

In \cite{AHMMS} Lemma~\ref{lem:colon-ideal-generation} is stated over the field $\RR$ of real numbers, but its proof goes 
through over any field.

\subsection{Arrangement combinatorics}
Let $V$ be an $n$-dimensional $\KK$-vector space with dual space 
$V^* = \Hom_\KK(V,\KK)$.
Fixing a basis $\ee_1, \dots, \ee_n$ of $V$ yields a dual basis $x_1, \dots, x_n$ of $V^*$
and identifications $\KK[V] = \KK[x_1, \dots, x_n] = S$.

As in the introduction, a {\em (linear) hyperplane} in $V$ is a codimension 1 linear subspace $H \subset V$.
If $\alpha \in V^* - \{0\}$, we write $H_\alpha$ for the hyperplane given by the kernel of $\alpha: V \rightarrow \KK$.
Conversely, a hyperplane $H \subset V$ is the kernel of a unique $\alpha \in V^* - \{0\}$, up to a nonzero scalar. 

An {\em arrangement} in $V$ is a finite set $\AAA$ of hyperplanes in $V$. The arrangement $\AAA$ is {\em essential} if 
\begin{equation}
\bigcap_{H \, \in \, \AAA} H = \{ 0 \}.
\end{equation}
A {\em flat} of $\AAA$ is an intersection $X \subseteq V$ of hyperplanes in $\AAA$ 
(including the `empty intersection' $V$).
The {\em intersection poset} $\LL(\AAA)$ is the collection of flats in $\AAA$ ordered by reverse containment:
we have $X \leq_{\LL(\AAA)} Y$ if and only if $Y \subseteq X$.
Since we are considering linear arrangements only, the poset $\LL(\AAA)$ is a lattice with minimum 
$\hat{0} = V$ and maximum  $\hat{1} = \bigcap_{H \in \AAA} H$.
The arrangement $\AAA$ is {\em supersolvable} if $\LL(\AAA)$ admits a maximal chain of modular elements (see \cite{Stanley}).
In particular, supersolvability is a {\em combinatorial property} of $\AAA$: it only depends on the intersection lattice $\LL(\AAA)$.

Let $P$ be a finite poset.  The {\em M\"obius function} $\mu_P: P \times P \rightarrow \ZZ$ is defined recursively by
\begin{equation}
\begin{cases}
\mu_P(x,x) = 1 & \text{for all $x \in P$,} \\
\mu_P(x,y) = 0 & \text{if $x \not\leq_P y$,}
\end{cases} \quad \text{and} \quad
\sum_{x  \, \leq_P \, z  \, \leq_P \, y} \mu_P(x,z) = 0
\text{ for all $x <_P y$}.
\end{equation}
If $\AAA$ is an arrangement, the {\em characteristic polynomial} of $\AAA$ is given by
\begin{equation}
\chi_\AAA(t) := \sum_{X \, \in \, \LL(\AAA)} \mu_{\LL(\AAA)}(\hat{0},X) \cdot t^{\dim X}.
\end{equation}

If $\AAA$ is an arrangement in $V$ and $H \in \AAA$ is a hyperplane, we define two new arrangements as follows.
The {\em deletion} $\AAA \setminus H$ is the subarrangement obtained from $\AAA$ by removing the hyperplane $H$.
The {\em restriction} $\AAA^H$ is the arrangement in $H$ given by
\begin{equation}
\AAA^H = \{ H' \cap H \,:\, H' \in \AAA, \, H' \neq H \}.
\end{equation}

\subsection{Derivation modules}
A {\em derivation} of $S$ is a $\KK$-linear map $\theta: S \rightarrow S$ which satisfies the Leibniz rule:
$\theta(fg) = f \theta(g) + g \theta(f)$ for all $f, g \in S$.
The set $\Der(S)$ of derivations $\theta: S \rightarrow S$ has the structure of an $S$-module via $(f \theta)(g) = f \cdot \theta(g)$ for $f, g \in S$.
For each $1 \leq i \leq n$, the partial derivative operator $\partial_i: S \rightarrow S$ is a derivation, and the set 
$\{ \partial_1, \dots, \partial_n \}$ forms a free $S$-basis for $\Der(S)$.
A derivation $\theta \in S$ is {\em homogeneous of degree $d$} if $\theta = f_1 \partial_1 + \cdots + f_n \partial_n$ where each 
$f_1, \dots, f_n \in S$ is homogeneous of degree $d$.

Let $\AAA$ be a hyperplane arrangement in $V = \KK^n$ with hyperplanes $H_1, \dots, H_m$ given by the kernels of the linear forms 
$\alpha_1, \dots, \alpha_m \in V^*$.
The {\em defining polynomial} of $\AAA$ is the product $Q(\AAA) := \prod_{i = 1}^m \alpha_i \in S$.  The polynomial $Q(\AAA)$ is defined up to
a nonzero scalar.
A derivation $\theta \in \Der(S)$ of $S$ is called a {\em derivation of $\AAA$} if $\theta(Q(\AAA))$ is divisible by $Q(\AAA)$.
The {\em derivation module} of $\AAA$ is the set $\Der(\AAA)$ of all derivations of $\AAA$; this is an $S$-submodule of $\Der(S)$.
For any arrangement $\AAA$ we have
\begin{equation}
\Der(\AAA) = \bigcap_{H \, \in \, \AAA} \Der(H).
\end{equation}
In particular, if $\BBB \subseteq \AAA$ is a subarrangement of $\AAA$ then $\Der(\AAA) \subseteq \Der(\BBB)$.
The arrangement $\AAA$ is {\em free} if $\Der(\AAA)$ is a free $S$-module.

If $\AAA$ is a free arrangement, the module $\Der(\AAA)$ will admit a homogeneous $S$-basis $\{ \rho_1, \dots, \rho_n \}$ where the multiset of 
degrees $e_i := \deg \rho_i$ is uniquely determined.  
The degrees $e_1, \dots, e_n$ are called the {\em exponents} of $\AAA$.
Terao's Factorization Theorem \cite{Terao} asserts that the characteristic polynomial of a free arrangement factors over $\ZZ$ as
\begin{equation}\label{eq:Terao's-Factorization-Thm}
\chi_\AAA(t) = \prod_{i \, = \, 1}^n (t - e_i).
\end{equation}

Showing that the characteristic polynomial of an arrangement does not factor over $\ZZ$ is often a practical way to show 
that an arrangement is not free, but there exist nonfree arrangements $\AAA$ for which $\chi_\AAA(t)$ factors over $\ZZ$.
On the other hand, it is known that every supersolvable arrangement is free. However, there are free arrangements which are not supersolvable.

\section{Solomon--Terao algebras}
\label{ST}

An $S$-module map $\ccc: \Der(S) \rightarrow S$ is determined by its images
$\ccc(\partial_i) \in S$ on the basis elements $\partial_i$ of $\Der(S)$ for $1 \leq i \leq n$.
We say that $\ccc$ is {\em homogeneous of degree $d$} if each $\ccc(\partial_i) \in S$ is homogeneous of degree $d$.
As in the introduction, if $\AAA$ is an arrangement in $\KK^n$ we write $\ccc_\AAA := \ccc(\Der(\AAA))$ for the image of $\Der(\AAA) \subseteq \Der(S)$
under $\ccc$.
If the map $\ccc$ is homogeneous, then $\ccc_\AAA \subseteq S$ is a homogeneous ideal.
We only consider homogeneous maps $\ccc$.
As defined in the introduction,
the Solomon--Terao algebra is $\ST(\AAA,\ccc) := S/\ccc_\AAA$.

Solomon--Terao algebras have arisen `naturally' in the context of Hessenberg theory \cite{AHMMS} with the $S$-module
map $\aaa: \Der(S) \rightarrow S$ given by $\aaa: \partial_i \mapsto x_i$.
For our applications to superspace coinvariant theory
we will consider Solomon--Terao algebras with respect to the map $\ii: \Der(S) \rightarrow S$ defined by $\ii: \partial_i \mapsto 1$ for all $i$.
In this section we establish some simple facts about Solomon--Terao algebras in general.

\begin{remark}
\label{rmk:ammn}
In earlier work,
Abe, Maeno, Murai, and Numata  defined \cite{AMMN} a class of quotients of $S$ called ``Solomon--Terao algebras."
Their definition depends on a hyperplane arrangement $\AAA$ in $\KK^n$ and a homogeneous polynomial $h \in S$.
Abe et. al. define the Solomon--Terao algebra to be the quotient of $S$ by the ideal $\{ \theta(h) \,:\, \theta \in \Der(\AAA) \}$.
This is equivalent to considering the $S$-module map $\ccc: \Der(S) \rightarrow S$ given by $\partial_i \mapsto \partial_i(h)$ for all $i$.
Our algebras $\ST(\AAA,\ccc)$ are therefore slightly more general than those in \cite{AMMN}.
However, both examples  $\ST(\AAA,\aaa)$ and $\ST(\AAA,\ii)$ of Solomon--Terao algebras which have arisen in applications
arise from the construction of Abe et.\ al., using the polynomials $h = \frac{1}{2}(x_1^2 + \cdots + x_n^2)$ and $h = x_1 + \cdots + x_n$.
\end{remark}

Solomon--Terao algebras behave well in the context of free arrangements.
To begin, we record a well-known lemma relating the exponents of a free arrangement
to the number of hyperplanes.

\begin{lemma}
\label{lem:free-arrangement-exponents}
Let $\AAA$ be a free arrangement with exponents $e_1, \dots, e_n$.
Then $\AAA$ contains $e_1 + \cdots + e_n$ hyperplanes.
\end{lemma}

\begin{proof}
By (\ref{eq:Terao's-Factorization-Thm}), the characteristic polynomial of $\AAA$ factors as 
\begin{equation}
\chi_\AAA(t) = \sum_{X \, \in \, \LL(\AAA)} \mu_{\LL(\AAA)}(\hat{0},X) \cdot t^{\dim(X)} =  \prod_{i \, = \, 1}^n (t - e_i).
\end{equation}
Comparing the coefficient of $t^{n-1}$ on both sides yields the result.
\end{proof}

We next show that,
if $\AAA$ is free and the Solomon--Terao algebra $\ST(\AAA,\ccc)$ has finite positive dimension over $\KK$, then $\ST(\AAA,\ccc)$ is a Poincar\'e duality algebra and its Hilbert series can be described by exponents.

\begin{lemma}
\label{lem:st-hilbert-series}
Let $\ccc: \Der(S) \rightarrow S$ be an $S$-module map of homogeneous degree $d \geq 0$ and let $\AAA$ be a free
arrangement in $\KK^n$ with exponents $e_1, \dots, e_n$.
Exactly one of the following items about the Solomon--Terao algebra $\ST(\AAA,\ccc)$ is true.
\begin{enumerate}
\item  We have $\ST(\AAA,\ccc) = 0$.
\item  The algebra $\ST(\AAA,\ccc)$ has infinite dimension over $\KK$.
\item The ideal $\ccc_\AAA \subseteq S$ is a complete intersection and the algebra $\ST(\AAA,\ccc)$ is a Poincar\'e duality algebra
with Hilbert series
$$
\Hilb(\ST(\AAA,\ccc);q) = \frac{(1-q^{e_1 + d}) \cdots (1 - q^{e_n + d})}{(1-q)^n}
$$
and socle degree $e_1 + \cdots + e_n + n(d-1) = |\AAA| + n(d-1)$.
\end{enumerate}
\end{lemma}

\begin{proof}
Let $\rho_1, \dots, \rho_n$ be a free basis for $\Der(\AAA)$ with $\deg \rho_i = e_i$.
The ideal $\ccc_\AAA \subseteq S = \KK[x_1, \dots, x_n]$ is generated by 
the $n$ elements $\ccc(\rho_1), \dots, \ccc(\rho_n) \in S$.
If (1) and (2) fail to hold, then $\left(\ccc(\rho_1), \dots, \ccc(\rho_n)\right) \neq S$ and $S / \ccc_\AAA$ is Artitinan. It follows from the discussion in Section ~\ref{sec:commutative-algebra} that $\ccc_\AAA$ is a complete intersection so $S / \ccc_{\AAA} = \ST(\AAA, \ccc)$ is a Poincar\'e Duality algebra.
The claimed Hilbert series follows from \eqref{eq:2.7} because $\deg \ccc(\rho_i) = e_i + d$.
The final equality in (3) is a consequence of Lemma~\ref{lem:free-arrangement-exponents}.
\end{proof}

Let us show that conditions (1) and (2) in Lemma~\ref{lem:st-hilbert-series} can occur.
Consider the map $\ii: \Der(S) \rightarrow S$ given by $\ii: \partial_i \mapsto 1$ for all $i$.
The Solomon--Terao algebra of the empty arrangement $\varnothing$ in $\KK^n$ is
$$
\ST(\varnothing,\ii) = S / \ii_\varnothing = S / (1) = 0.
$$
On the other hand, if $n = 2$ and $\AAA = \{x_1 + x_2 = 0\}$, it is not hard to see that 
$\AAA$ is a free arrangement and the derivation module of $\AAA$ has $S$-basis
$\{ \partial_1 - \partial_2, x_2 \partial_1 - x_1 \partial_2 \}$.
Since $\ii(\partial_1 - \partial_2) = 1 - 1 = 0$ and
$\ii(x_2 \partial_1 - x_1 \partial_2) = x_2 - x_1$, we see that
$$
\ST(\AAA,\ii) = \KK[x_1, x_2] / \ii_\AAA = \KK[x_1, x_2] / (x_2 - x_1)
$$
is an infinite-dimensional $\KK$-vector space.

There are conditions on $\ccc$ which guarantee that (1) and (2) in Lemma~\ref{lem:st-hilbert-series} never occur.
If $d > 0$, the ideal $\ccc_\AAA \subseteq S$ lives in strictly positive degrees so that (1) is impossible.
Degree reasons also forbid (1) when $e_1, \dots, e_n > 1$.
When $\ccc$ is the map $\aaa: \partial_i \mapsto x_i$ and $\KK = \RR$, (1) cannot occur for degree reasons and (2) cannot occur by work of Solomon--Terao \cite[Prop.\ 4.10, $h = \frac{1}{2}(x_1^2 + \cdots + x_n^2)$]{ST}
\footnote{In order to apply \cite[Prop.\ 4.10]{ST} one uses that $\frac{1}{2}(x_1^2 + \cdots + x_n^2)$ is a positive definite quadratic form on $\RR^n$. See \cite[Defn.\ 1.6]{AMMN} and \cite[Prop.\ 2.8 (1), $i = 0$, $\eta = \frac{1}{2}(x_1^2 + \cdots + x_n^2)$]{AMMN}  for a more direct statement using the language of Solomon--Terao algebras.}.
Since we are considering the degree 0 map $\ii: \partial_i \mapsto 1$, we will need to be mindful
of the full trichotomy in Lemma~\ref{lem:st-hilbert-series}.

Let $\AAA$ be a free arrangement and let $\BBB \subseteq \AAA$ be a subarrangement which is also free.
Since $\Der(\AAA) \subseteq \Der(\BBB)$, the Solomon--Terao algebra $\ST(\AAA,\ccc)$ projects onto $\ST(\BBB,\ccc)$.
Under mild hypotheses, we can be more precise about the nature of this projection as follows.
Recall that $Q(\AAA), Q(\BBB) \in S$ are the defining polynomials of $\AAA$ and $\BBB$. Since $\BBB \subseteq \AAA$,
we have $Q(\BBB) \mid Q(\AAA)$.

\begin{lemma}
    \label{lem:derivation_colon}
    Let $\AAA$ be a free arrangement in $\KK^n$
    and let $\BBB \subseteq \AAA$ be a free 
    subarrangement. 
    Let $\ccc: \Der(S) \rightarrow S$ be a homogeneous map.
    Assume the following.
    \begin{enumerate}
        \item The big Solomon--Terao algebra 
        $\ST(\AAA,\ccc)$ is a finite-dimensional $\KK$-vector space.
        \item The quotient
        $Q(\AAA)/Q(\BBB)$ does not
        lie in $\ccc_\AAA$.
    \end{enumerate}
    Then $\ccc_\BBB = \ccc_\AAA : (Q(\AAA)/Q(\BBB))$
    as ideals in $S$.
\end{lemma}
\begin{proof}
It follows from the definitions that multiplication by $Q(\AAA)/Q(\BBB)$ carries $\Der(\BBB)$ into $\Der(\AAA)$.
Applying the map $\ccc$ shows that 
$\ccc_\BBB \subseteq \ccc_\AAA : (Q(\AAA)/Q(\BBB))$ as ideals in $S$.
This also shows $\ST(\AAA,\ccc)$ and $\ST(\BBB,\ccc)$ are nonzero since  assumption (2) tells that $\ccc_\AAA : (Q(\AAA)/Q(\BBB))$ is not a unit ideal.
Assumption (1), the fact that $\ccc_\AAA \subseteq \ccc_\BBB$, and Lemma~\ref{lem:st-hilbert-series}
imply that $\ST(\AAA,\ccc)$ and $\ST(\BBB,\ccc)$ are Poincar\'e duality algebras of socle degrees $|\AAA| + n(d-1)$
and $|\BBB| + n(d-1)$, respectively, where $d$ is the degree of $\ccc$.  
Since 
\begin{equation}
\deg(Q(\AAA)/Q(\BBB)) = \deg Q(\AAA) - \deg Q(\BBB) = |\AAA| - |\BBB| = ( |\AAA| + n(d-1)) - (|\BBB| + n(d-1)),
\end{equation}
assumption (2) gives us license to apply Lemma~\ref{lem:colon-ideal-generation} to prove the result.
\end{proof}

Lemma~\ref{lem:derivation_colon} realizes certain Solomon--Terao ideals as colon ideals.
We will show that a family of ideals closely related to the colon ideals $(S^{\symm_n}_+) : f_J$ 
appearing in Theorem~\ref{thm:rw} fits into the framework of Lemma~\ref{lem:derivation_colon}.

\section{Southwest arrangements}
\label{Southwest}

\subsection{Southwest combinatorics}
Theorem~\ref{thm:rw} gives a transfer principle between the superspace coinvariant ring $SR = \Omega/I$ and quotients 
of the classical polynomial ring $S$ by colon ideals ${(S^{\symm_n}_+) : f_J}$ where each $f_J$ is a product of $(x_i - x_j)$'s
and $x_i$'s.  
We want to put these latter commutative quotients in a context amenable to Lemma~\ref{lem:derivation_colon}.
For this, it is natural to consider hyperplanes given by the kernels of linear forms in $\widetilde{\Phi}^+$.
This motivates the following definition.

For $1 \leq i < j \leq n$, write $H_{i,j} \subseteq \KK^n$ for the hyperplane $x_i = x_j$.
For $1 \leq j \leq n$, write 
\[H_{0,j} = H_j = H_{x_j} \] 
 for the coordinate hyperplane $x_j = 0$.
Let $\AAA_{\widetilde{\Phi}^+} = \{ H_{i,j} \,:\, 1 \leq i < j \leq n \} \cup \{ H_j \,:\, 1 \leq j \leq n \}$.
Our notational conventions yield
$\AAA_{\widetilde{\Phi}^+} = \{ H_{i,j} \,:\, 0 \leq i < j \leq n \}$.

\begin{defn}
\label{defn:sw-arrangement}
A subarrangement $\AAA \subseteq \AAA_{\widetilde{\Phi}^+}$ is a {\em southwest arrangement}
if it satisfies the following condition:
\begin{center}
$H_{i,j} \in \AAA$ and $j > i+1$ imply $H_{i,j-1} \in \AAA$.
\end{center}
\end{defn}

The ``southwest" terminology is motivated by the following picture. We give the linear forms in $\widetilde{\Phi}^+$ 
the partial order $(x_a - x_d) \geq (x_b - x_c)$ if $0 \leq a \leq b < c \leq d \leq n$, where we identify $x_d \leftrightarrow x_0 - x_d$.
This gives rise to an augmented version of the type A positive root poset, shown here in the case $n = 5$.

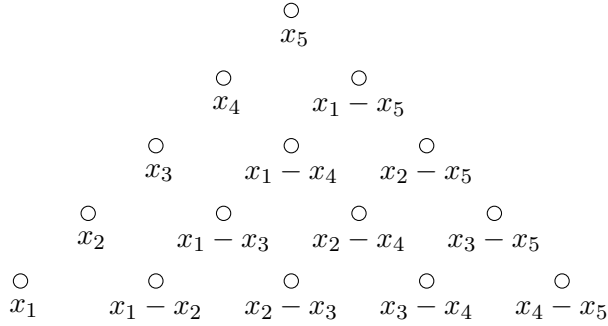
\vspace{0.1in}

\begin{center}
\begin{tikzpicture}[scale = 0.9]

\draw (0,0) circle (3pt);
\node at (0.05,-0.4) {$x_1$};

\draw (2,0) circle (3pt);
\node at (2,-0.4) {$x_1 - x_2$};

\draw (4,0) circle (3pt);
\node at (4,-0.4) {$x_2 - x_3$};

\draw (6,0) circle (3pt);
\node at (6,-0.4) {$x_3 - x_4$};

\draw (8,0) circle (3pt);
\node at (8,-0.4) {$x_4 - x_5$};

\draw (1,1) circle (3pt);
\node at (1.05,0.6) {$x_2$};

\draw (3,1) circle (3pt);
\node at (3,0.6) {$x_1 - x_3$};

\draw (5,1) circle (3pt);
\node at (5,0.6) {$x_2 - x_4$};

\draw (7,1) circle (3pt);
\node at (7,0.6) {$x_3 - x_5$};

\draw (2,2) circle (3pt);
\node at (2.05,1.6) {$x_3$};

\draw (4,2) circle (3pt);
\node at (4,1.6) {$x_1 - x_4$};

\draw (6,2) circle (3pt);
\node at (6,1.6) {$x_2 - x_5$};

\draw (3,3) circle (3pt);
\node at (3.05,2.6) {$x_4$};

\draw (5,3) circle (3pt);
\node at (5,2.6) {$x_1 - x_5$};

\draw (4,4) circle (3pt);
\node at (4.05,3.6) {$x_5$};

\end{tikzpicture}
\end{center}

A subarrangement of $\AAA_{\widetilde{\Phi}^+}$ is southwest if its hyperplanes are closed under moving southwest in this poset.
This is a weaker condition than being an order ideal of $\widetilde \Phi^+$.
When $n = 5$, the southwest arrangement
\begin{equation}\label{eqn:ex-hyperplane}
   \AAA = \{ H_1, H_2, H_{1,2}, H_{1,3}, H_{2,3}, H_{1,4}, H_{2,4}, H_{3,4}, H_{2,5} \} 
\end{equation}
may be illustrated as follows, where hyperplanes correspond to black dots.

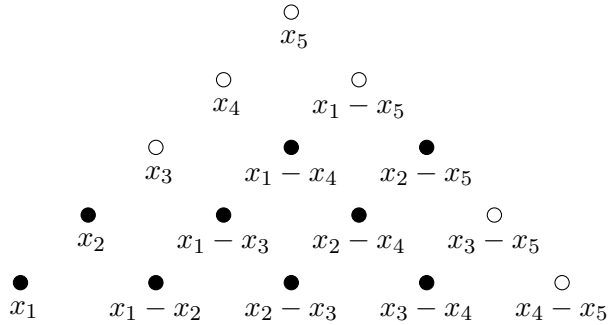
\vspace{0.1in}

\begin{center}
\begin{tikzpicture}[scale = 0.9]

\filldraw (0,0) circle (3pt);
\node at (0.05,-0.4) {$x_1$};

\filldraw (2,0) circle (3pt);
\node at (2,-0.4) {$x_1 - x_2$};

\filldraw (4,0) circle (3pt);
\node at (4,-0.4) {$x_2 - x_3$};

\filldraw (6,0) circle (3pt);
\node at (6,-0.4) {$x_3 - x_4$};

\draw (8,0) circle (3pt);
\node at (8,-0.4) {$x_4 - x_5$};

\filldraw (1,1) circle (3pt);
\node at (1.05,0.6) {$x_2$};

\filldraw (3,1) circle (3pt);
\node at (3,0.6) {$x_1 - x_3$};

\filldraw (5,1) circle (3pt);
\node at (5,0.6) {$x_2 - x_4$};

\draw (7,1) circle (3pt);
\node at (7,0.6) {$x_3 - x_5$};

\draw (2,2) circle (3pt);
\node at (2.05,1.6) {$x_3$};

\filldraw (4,2) circle (3pt);
\node at (4,1.6) {$x_1 - x_4$};

\filldraw (6,2) circle (3pt);
\node at (6,1.6) {$x_2 - x_5$};

\draw (3,3) circle (3pt);
\node at (3.05,2.6) {$x_4$};

\draw (5,3) circle (3pt);
\node at (5,2.6) {$x_1 - x_5$};

\draw (4,4) circle (3pt);
\node at (4.05,3.6) {$x_5$};

\end{tikzpicture}
\end{center}

For any southwest arrangement $\AAA$, we associate an {\em $h$-sequence}
$h(\AAA) = (h_1, \dots, h_n)$ by 
\begin{equation}
h_i = h_i(\AAA) := | \AAA \cap \{ H_i, H_{1,i}, H_{2,i}, \dots, H_{i-1,i} \} |.
\end{equation}
In terms of the poset structure on $\widetilde{\Phi}^+$, the number $h_i$ counts black dots in the $i^{th}$ northwest-to-southeast diagonal.
The above example has $h(\AAA) = (h_1, \dots, h_5) = (1,2,2,3,1)$.
 It will turn out that $(h_1, \dots, h_n)$ are the exponents of a southwest arrangement $\AAA$. 
 
The class of southwest arrangements generalize a family of arrangements which arise in type A Hessenberg theory 
\cite{AHHM, AHMMS} corresponding to order ideals $\III \subseteq \Phi^+$ of the type A positive root poset.
Such ideal arrangements are in one-to-one correspondence with {\em Hessenberg functions} 
$h: [n] \rightarrow [n]$ which satisfy $h(1) \leq \cdots \leq h(n)$ and $h(i) \geq i$ for all $i$;
this motivates our use of $h$.\footnote{One also needs to replace $h(i) \mapsto h(i) - i$.}
Unlike in the Hessenberg case, there can be multiple southwest arrangements $\AAA$ with the same $h$-sequence.

Southwest arrangements and their $h$-functions behave nicely under deletion of and restriction to certain coordinate hyperplanes $H_{p}$.
We define southwest arrangements in $H_{x_p}$ in the natural way using the ordered basis $\ee_1, \dots, \ee_{p-1}, \ee_{p+1}, \dots, \ee_n$
of $H_{p}$. We also use the abbreviation
\begin{equation}
\AAA^{(p)} := \AAA^{H_{p}}
\end{equation}
for the restriction of an arrangement $\AAA$ to the $p^{th}$ coordinate hyperplane $H_p$.

\begin{lemma}
\label{lem:sw-recursion}
Let $\AAA$ be a southwest arrangement in $\KK^n$ with 
$h(\AAA) = (h_1, \dots, h_n)$. Assume that $\AAA$ contains at least one coordinate hyperplane
and write  $p := \max \{ k \,:\, H_{k} \in \AAA \}$.
Then
\begin{enumerate}
\item $\AAA \setminus H_{p}$ is a southwest arrangement with 
\begin{equation*}
h(\AAA \setminus H_{p}) = (h_1, \dots ,h_{p-1}, h_p - 1, h_{p+1}, \dots, h_n).
\end{equation*}
\item $\AAA^{(p)}$ is a southwest arrangement in $H_{p}$ with
\begin{equation*}
h(\AAA^{(p)}) = (h_1, \dots, h_{p-1}, h_{p+1}, \dots, h_n).
\end{equation*}
\end{enumerate}
\end{lemma}

\begin{proof}
 (1) is clear from the definitions. For  (2), restriction to $H_{x_p}$ has the effect of setting $x_p \rightarrow 0$ in the linear forms
defining the hyperplanes of $\AAA$. This leaves linear forms $x_i - x_j$ unchanged if $i, j \neq p$ while carrying $x_i - x_p$ to $x_i$.
In the poset picture of $\widetilde{\Phi}^+$, the hyperplanes involving $x_p$ form a V-shaped set.
Restriction to $H_p$ eliminates the dots in the northwest-to-southeast portion of the V while moving the black dots in the 
southwest-to-northeast portion of the V to the top of their northwest-to-southeast diagonals. This process preserves the southwest condition
and has the claimed effect on $h$-sequences.
\end{proof}

Lemma~\ref{lem:sw-recursion} and its proof are best understood by example.
If $n = 5$ and $\AAA$ is the southwest
arrangement considered before (see \eqref{eqn:ex-hyperplane}), we have $p = 2$. The hyperplane $H_2$ has ordered basis $\ee_1, \ee_3, \ee_4, \ee_5$ with coordinate functions
$x_1, x_3, x_4, x_5$ and $\AAA^{(2)}$ is shown below.
Observe that $h(\AAA^{(2)}) = (1,2,3,1)$.

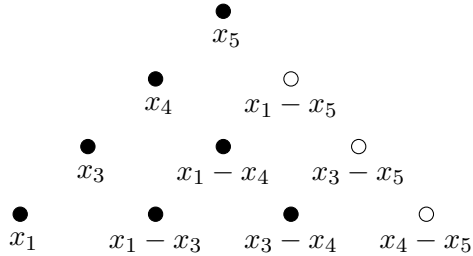
\vspace{0.1in}

\begin{center}
\begin{tikzpicture}[scale = 0.9]

\filldraw (1,1) circle (3pt);
\node at (1.05,0.6) {$x_1$};

\filldraw (5,1) circle (3pt);
\node at (5,0.6) {$x_3 - x_4$};

\draw (7,1) circle (3pt);
\node at (7,0.6) {$x_4 - x_5$};

\filldraw (3,1) circle (3pt);
\node at (3,0.6) {$x_1 - x_3$};

\draw (6,2) circle (3pt);
\node at (6,1.6) {$x_3 - x_5$};

\filldraw (2,2) circle (3pt);
\node at (2.05,1.6) {$x_3$};

\filldraw (4,2) circle (3pt);
\node at (4,1.6) {$x_1 - x_4$};

\filldraw (3,3) circle (3pt);
\node at (3.05,2.6) {$x_4$};

\draw (5,3) circle (3pt);
\node at (5,2.6) {$x_1 - x_5$};

\filldraw (4,4) circle (3pt);
\node at (4.05,3.6) {$x_5$};

\end{tikzpicture}
\end{center}

\subsection{Southwest algebra}
Solomon--Terao algebras coming from southwest arrangements are deeply tied to the superspace coinvariant ring
$SR$. Lemma~\ref{lem:st-hilbert-series} says that Solomon--Terao algebras have nice properties for free arrangements.
Southwest arrangements are not merely free, but also supersolvable.

\begin{proposition}
\label{prop:sw-free}
Any southwest arrangement is supersolvable, hence free.
\end{proposition}

\begin{proof}
Give $\KK^{n+1}$ the coordinates $x_0, x_1, \dots, x_n$.
For a simple graph $\Gamma$ on the vertex set $\{0, 1, \dots, n \}$, 
we have a {\em graphical arrangement} $\AAA_\Gamma$ in $\KK^{n+1}$ formed from the hyperplanes 
$x_i - x_j = 0$ for each edge $i-j$ of $\Gamma$.
If $\AAA$ is a southwest arrangement in $\KK^n$, we have a graphical arrangement $\AAA'$ in $\KK^{n+1}$
obtained by replacing each coordinate hyperplane $H_k$ in $\AAA$ with the hyperplane $x_0 - x_k$.
It is easy to see that $\AAA$ and ${\AAA}'$ have isomorphic intersection posets:
\[\LL(\AAA) \cong \LL({\AAA}').\]
The definition of supersolvability depends only on the intersection poset. In particular, $\AAA$ is supersolvable if and only if 
${\AAA}'$ is supersolvable.

A graph $\Gamma$ on $\{0, 1, \dots, n \}$ is {\em chordal} if it has no induced $k$-cycles for $k > 3$.
If $\AAA_\Gamma$ is a graphical arrangement, Stanley proved (see \cite[Cor.\ 4.10]{Stanley} and the note following that result)
that 
\begin{center}
$\AAA_\Gamma$ is supersolvable $\Leftrightarrow$ 
$\Gamma$ is chordal.
\end{center}
It is not hard to check that if $\AAA$ is a southwest arrangement, the graphical arrangement $\AAA'$ corresponds to a chordal
graph $\Gamma$.
Since every supersolvable arrangement is free, we are done.
\end{proof}

Thanks to the correspondence between subarrangements of $\AAA_{\widetilde{\Phi}^+}$ and 
graphical arrangements $\AAA_\Gamma$ for graphs $\Gamma$ on $\{0,1,\dots,n\}$,
Proposition~\ref{prop:sw-free} can be augmented. We have the following result on arbitrary subarrangements of $\AAA_{\Phi^+}$ and $\AAA_{\widetilde{\Phi}^+}$.


\begin{proposition}
\label{prop:graphical-properties}
Let $\Gamma$ be a graph on the vertex set $\{0,1,\dots,n\}$ and let $\AAA_\Gamma$ be the corresponding graphical arrangement in $\KK^{n+1}$.
Let $\AAA'_\Gamma \subseteq \AAA_{\widetilde{\Phi}^+}$ be the arrangement obtained from $\AAA_\Gamma$ by restriction to $x_0 = 0$.
The following are equivalent.
\begin{enumerate}
\item $\Gamma$ is chordal.
\item $\AAA_\Gamma$ is supersolvable.
\item $\AAA_\Gamma$ is free.
\item The Orlik-Solomon algebra $\OS(\AAA_{\Gamma})$ is Koszul.
\item The Orlik-Solomon ideal $I_{\AAA_\Gamma}$ is quadratic.
\end{enumerate}
Properties (1) - (5) are also equivalent for the arrangement $\AAA'_\Gamma$.
\end{proposition}

Proposition~\ref{prop:graphical-properties} is a collection of known results from the hyperplane arrangements literature.
We give the relevant references here.

\begin{proof}
For the equivalence of (1)-(3) for the graphical arrangement $\AAA_\Gamma$, see \cite[Thm.\ 3.3]{ER} and its proof.
If $\AAA$ is any arrangement we have 
\begin{center}
$\AAA$ is supersolvable $\Rightarrow$ $\AAA$ is free.
\end{center}
Since supersolvability depends only on the intersection poset and since $\AAA_\Gamma, \AAA'_\Gamma$
have the same intersection posets, we see that (1)-(3) are equivalent for $\AAA'_\Gamma$, as well.
\end{proof}

\begin{remark}
    It follows from work of Schenk and Suciu \cite[Thm.\ 6.4]{SchenckSuciu} that the equivalent conditions $(1)$-$(3)$ of Proposition~\ref{prop:graphical-properties} are also equivalent to (4) the Koszulality of the Orlik--Solomon algebra of $\AAA_\Gamma$ (or of $\AAA'_\Gamma$), as well as (5) the Orlik--Solomon ideal of $\AAA_\Gamma$ (or of $\AAA'_\Gamma$) beng quadratic.
\end{remark}

\subsection{Free bases for southwest arrangements}
In order to understand Solomon--Terao algebras coming from southwest arrangements, we will use an explicit free basis of their derivation modules.
Such a basis may be given as follows.  Let $\AAA \subseteq \AAA_{\widetilde{\Phi}^+}$ be a southwest 
arrangement. For $0 \leq i <j \leq n$, we define $\alpha_{i,j}=x_i-x_j$ if $ i \ne 0$ and $\alpha_{0,j}=x_j$. Define elements $\rho^\AAA_j \in \Der(S)$ for $1 \leq j \leq n$ by the formula
\begin{equation}
\rho^\AAA_j := \sum_{k \, = \, j}^n \left(  \prod_{H_{i,j} \, \in \, \AAA} \alpha_{i,k} \right) \partial_k
\end{equation}
where $\rho^\AAA_j = \partial_j + \partial_{j+1} + \cdots + \partial_n$ when 
$\AAA \cap \{ H_{i,j} \,:\, i = 0, 1, \dots, j-1 \} = \varnothing$.
Observe that the degree of $\rho^\AAA_j$ is given by the $j^{th}$ entry $h_j$ of the $h$-sequence
$h(\AAA) = (h_1, \dots, h_n)$. As an example, if $\AAA$ is the arrangement from (\ref{eqn:ex-hyperplane}), then $\rho^\AAA_4 = \alpha_{1, 4}\alpha_{2, 4}\alpha_{3,4}\partial_4 + \alpha_{1, 5}\alpha_{2, 5}\alpha_{3,5}\partial_5.$  

\begin{proposition}
\label{prop:derivation-module-basis}
If $\AAA$ is a southwest arrangement in $\KK^n$, then $\Der(\AAA)$ is a free $S$-module with basis
$\rho^\AAA_1, \dots, \rho^\AAA_n$. In particular, the exponents of $\AAA$ are given by $h(\AAA) = (h_1, \dots, h_n)$.
\end{proposition}

\begin{proof}
We abbreviate $\rho_j := \rho^\AAA_j$.
Saito's Criterion \cite{Saito} states that $\rho_1, \dots, \rho_n$  is a free $S$-basis of $\Der(\AAA)$ 
if the following conditions are satisfied:
\begin{enumerate}
\item we have $\rho_j \in \Der(\AAA)$ for each $j$,
\item $\deg \rho_1 + \cdots + \deg \rho_n = |\AAA|$,
\item $\rho_1, \dots, \rho_n$ are linearly independent over $S$.
\end{enumerate}
 (2) follows from the definition of the $\rho_j$.
 (3) follows from the fact that each $\rho_j$ is an $S$-linear combination of $\partial_j,\dots,\partial_n$
where the coefficient of $\partial_j$ is nonzero.
We prove (1) as follows.

Fix $1 \leq j \leq n$ and $\alpha_{s,t} \in \mathcal A$.
We prove $\alpha_{s,t} \mid \rho_j(\alpha_{s,t})$ in the following case-by-case fashion.
\medskip

{\bf Case 1:}
{\em Suppose $t <j$.} 

Then $\rho_j (\alpha_{s,t}) = 0$ so that $\alpha_{s,t} \mid \rho_j(\alpha_{s,t})$.

{\bf Case 2:}
{\em Suppose $s < j \leq t$.} 

Since $H_{s,t} \in  \AAA$ implies $H_{s,j} \in  \AAA$ by the southwest condition and
\[\rho_j (\alpha_{s,t}) = \prod_{H_{i,j} \, \in \, \mathcal A} \alpha_{i,t}\]
by direct computation, we have $\alpha_{s,t} \mid \rho_j(\alpha_{s,t})$.

{\bf Case 3:}
{\em Suppose $j \leq s$.} 

We compute 
\[ \rho_j (\alpha_{s,t}) =\prod_{H_{i,j} \, \in \, \mathcal A} \alpha_{i,s} - \prod_{H_{i,j} \, \in  \, \mathcal A} \alpha_{i,t}.\]
This polynomial vanishes if we substitute $x_s=x_t$, which means that $\rho_j(\alpha_{s,t})$ is contained in the ideal $(x_s-x_t)= (\alpha_{s,t})$, as desired. 
\end{proof}

\begin{remark}
\label{rmk:ziegler}
In his Ph.D. thesis, Ziegler developed \cite{Ziegler} a systematic method for computing a free basis of the derivation module 
$\Der(\AAA)$ of a supersolvable arrangement $\AAA$ from an M-chain in the intersection lattice $\LL(\AAA)$.
By Proposition~\ref{prop:sw-free}, this method may be used to find a basis of $\Der(\AAA)$ for southwest arrangements $\AAA$.
Thanks to the work of Stanley \cite[Cor.\ 4.10]{Stanley} and Edelman--Reiner \cite[Thm.\ 3.3]{ER},
Ziegler's method would apply to give an explicit free basis for $\Der(\AAA)$ whenever $\AAA \subseteq \AAA_{\widetilde{\Phi}^+}$ 
is a subarrangement whose corresponding graph on $\{0,1,\dots,n\}$ is chordal.
The authors have not computed such a basis at this level of generality.
\end{remark}

\subsection{Exact sequences of derivation modules}
The form of the free basis in Proposition~\ref{prop:derivation-module-basis} will play an important role in our analysis.
Our inductive arguments will also require short exact sequences which connect these modules and the corresponding Solomon--Terao algebras
with respect to the map $\ii$.
These exact sequences arise in the general context of free arrangements, and are closely related to the following 
Addition-Deletion Theorem of Terao.

\begin{theorem}
\label{thm:addition}
{\em (Terao \cite{TeraoAddition})}
Let $\AAA$ be an arrangement in $\KK^n$ and let $H \in \AAA$. Let $e_1 \leq \cdots \leq e_n$ be a weakly increasing sequence of nonnegative integers and let $1 \leq k \leq n$. Any two of the following 
implies the third. 
\begin{enumerate}
\item $\AAA$ is free with exponent multiset $\{e_1, \dots, e_{k-1}, e_k, e_{k+1}, \dots , e_n\}$.
\item $\AAA \setminus H$ is free with exponent multiset $\{e_1, \dots, e_{k-1},  e_k - 1, e_{k+1}, \dots, e_n\}$.
\item $\AAA^H$ is free with exponent multiset $\{e_1, \dots, e_{k-1},  e_{k+1}, \dots, e_n\}$.
\end{enumerate}
\end{theorem}

Theorem~\ref{thm:addition} is tied 
to a short exact sequence of 
derivation modules.
Let $\AAA$ be any arrangement in $V = \KK^n$
and let $H = H_\alpha \in \AAA$ 
be a hyperplane corresponding 
to a linear form $\alpha$. Write
$\bar{S} := S/(\alpha)$ and regard $\bar{S}$ as the coordinate ring $\bar{S} = \KK[H]$ of the hyperplane $H$.
For any $\psi \in \Der(\AAA)$, we have 
$\alpha \mid \psi(\alpha)$ in $S$, so that 
we have a $\KK$-linear map
\begin{equation}
    \overline{\psi}: \bar{S} \longrightarrow \bar{S} 
    \quad \quad
    \overline{\psi}: f + (\alpha) \mapsto \psi(f) + 
    (\alpha)
\end{equation}
of the quotient ring $\bar{S}$.  

We claim that $\overline{\psi}$ is a derivation of the polynomial ring $\bar{S}$ for all $\psi \in \Der(\AAA)$. To see this, we use the following general definition of $\KK$-algebra derivations; see e.g. \cite[p.\ 385]{Eisenbud}. Let $B$ be a $\KK$-algebra and let $\rho: B \to B$ be a $\KK$-linear map. Then $\rho$ is a {\em derivation of $B$} if it satisfies the Leibniz rule $\rho(a b) = \rho(a) b + a \rho(b)$ for all $a,b \in B$. This coincides with our definition of derivations when $B$ is a polynomial ring over $\KK$; see e.g. \cite[Prop.\ 16.1]{Eisenbud}. Under this definition, the fact that $\overline{\psi} \in \Der(\bar{S})$ is inherited from $\psi \in \Der(S)$.

It turns out \cite[Prop.\ 4.44]{OT}
that $\overline{\psi} \in \Der(\AAA^H)$ for
all $\psi \in \Der(\AAA)$.\footnote{After a linear 
change of coordinates, we may assume that 
$\alpha = x_1$. One checks that 
$\beta \mid \overline{\psi}(\beta)$ for all linear 
forms
$\beta \in H^* = \Hom_\KK(H,\KK)$ corresponding to the hyperplanes
of $\AAA^H$.}
We therefore have a well-defined map
\begin{equation}
    \eta_H: \Der(\AAA) \longrightarrow \Der(\AAA^H)
    \quad \quad
    \eta_H: \psi \mapsto \overline{\psi}.
\end{equation}
Observe that $\eta_H$ preserves degree.

The map $\eta_H: \Der(\AAA) \rightarrow \Der(\AAA^H)$
has a geometric interpretation.
We may regard $\Der(S)$
as the $S$-module of polynomial vector fields on 
$V$. Under this identification, the derivation
module $\Der(\AAA) \subseteq \Der(S)$ consists 
of vector fields which are parallel to the 
hyperplanes of $\AAA$.  Restricting any such 
vector field to $H \subset V$ yields a vector 
field on $H$ which is parallel to the hyperplanes
of the restricted arrangement $\AAA^H$.
The map $\eta_H: \Der(\AAA) \rightarrow \Der(\AAA^H)$
is precisely this restriction.
The following short exact sequence
involving $\eta_H$ may be found in e.g. 
\cite[Thm.\ 4.46]{OT}.

\begin{theorem}
\label{thm:exact-triple}
The $\AAA$ be an arrangement and let $H \in \AAA$ be a hyperplane which satisfies the conditions of 
Theorem~\ref{thm:addition}. Let $\alpha$ be a linear form with kernel $H$.
We have a short exact sequence
\begin{equation}
0 \longrightarrow \Der(\AAA \setminus H) \xrightarrow{ \, \, \times \alpha \, \, } \Der(\AAA) 
\xrightarrow{\, \, \eta_H \, \, } \Der(\AAA^H) \longrightarrow 0
\end{equation}
where the first map is multiplication by $\alpha$ and the second map is homogeneous of degree 0.
\end{theorem}



Importantly for us, triples satisfying the hypotheses 
of Theorem~\ref{thm:exact-triple} arise from essential southwest arrangements as follows.
Observe that every essential subarrangement of $\AAA_{\widetilde{\Phi}^+}$ contains at least one coordinate hyperplane.

\begin{lemma}
\label{lem:sw-addition}
Let $\AAA$ be an essential southwest arrangement.
Define $p := \max \{ k \,:\, H_{k} \in \AAA \}$.  The triple $(\AAA, \AAA \setminus H_{p}, \AAA^{(p)})$ satisfies
the conditions of Theorem~\ref{thm:addition}.
\end{lemma}

\begin{proof}
This follows from Lemma~\ref{lem:sw-recursion} and Proposition~\ref{prop:sw-free}.
\end{proof}

\section{The algebra $S/((S_+^{\mathfrak S_n}):f_J)$ and the arrangement $\AAA_J$}
\label{AJ}

Recall that we are interested in a basis of the algebra $S/((S_+^{\mathfrak S_n}):f_J)$.
In this section, we show that this algebra is the Solomon--Terao algebra of a hyperplane arrangement,
and as an application we give an explicit generating set of the ideal $(S_+^{\mathfrak S_n}):f_J$. Doing this requires a lemma on when products over subsets  $T \subseteq \widetilde{\Phi}^+$ lie in the classical coinvariant ideal $(S^{\symm_n}_+)$. We write $V = \KK^n$ so that $\widetilde{\Phi}^+$ is a subset of $V^* = \Hom_\KK(V,\KK) =  \KK \cdot \{x_1, \dots, x_n\}$.

\begin{lemma}
    \label{lem:cospan_containment}
    Let $T \subseteq \widetilde{\Phi}^+$.
    We have $\prod_{\alpha \in T} \alpha \in (S^{\symm_n}_+)$  if and only if  the complement $\widetilde{\Phi}^+ - T$  of $T$ does not span $V^* = \KK \cdot \{x_1, \dots, x_n\}$.
\end{lemma}

The proof of Lemma~\ref{lem:cospan_containment} is technical and involves methods which are independent from the rest of the paper. As such, we relegate this proof to the appendix (Section~\ref{sec:appendix}).

Recall that $\ii: \Der(S) \rightarrow S$ is the $S$-module homomorphism defined by $\ii: \partial_i \mapsto 1$ for all $i$. The following result gives a link between Solomon--Terao algebras $\ST(\AAA,\ii)$ of subarrangements $\AAA \subseteq \AAA_{\widetilde{\Phi}^+}$ and coinvariant theory.


\begin{proposition}
\label{prop:st-general}
Let $\AAA \subseteq \AAA_{\widetilde{\Phi}^+}$ be an arbitrary subarrangement and let
\begin{equation*}
\beta_\AAA := \prod_{\substack{0 \, \leq \, i \, < \, j \, \leq \, n \\ H_{i,j} \, \notin \, \AAA}} \alpha_{i,j}.
\end{equation*}
\begin{enumerate}
\item 
The algebra $\ST(\AAA,\ii)$ is a quotient of the coinvariant ring $S/(S^{\symm_n}_+)$.
\item If $\AAA$ is not essential then $\ii_\AAA = S$ so that $\ST(\AAA,\ii) = 0$.
\item If $\AAA$ is a free arrangement then $\ii_\AAA = (S^{\symm_n}_+) : \beta_\AAA$.
\end{enumerate}
\end{proposition}

\begin{proof}
We start with the proof of (1). Recall that $\aaa: \Der(S) \rightarrow S$ is the $S$-module homomorphism defined by $\aaa: \partial_i \mapsto x_i$ for all $i$.
As mentioned in the introduction, we have $\aaa(\Der(\AAA_{\Phi^+})) = (S^{\symm_n}_+)$ (see \cite[Theorem 3.9]{AHMMS}). 
On the other hand, the form of the free bases in Proposition~\ref{prop:derivation-module-basis} for the southwest 
arrangements $\AAA = {\AAA}_{\Phi^+}, \AAA_{\widetilde{\Phi}^+}$
implies that 
\begin{equation}
\label{invariantii}
\ii(\Der({\AAA}_{\widetilde \Phi^+})) = \aaa(\Der(\AAA_{\Phi^+})) =  (S^{\symm_n}_+).
\end{equation}
Since $\AAA \subseteq \AAA_{\widetilde{\Phi}^+}$ implies $\Der(\AAA_{\widetilde{\Phi}^+}) \subseteq \Der(\AAA)$,
 (1) follows.

 For (2), adopt the shorthand $H_i := H_{x_i}$ and $H_{ij} := H_{x_i - x_j}$. We may assume without loss of generality that $H_i, H_{ij} \in \AAA$ implies $H_j \in \AAA$ and that $H_{ij}, H_{jk} \in A$ implies $H_{ik} \in \AAA$. Let $T = \{1 \leq i \leq n \,:\, H_i \in \AAA\}$. Then $T \neq [n]$ because $A$ is essential. We have an equivalence relation on $[n]-T$ defined by \[i \sim j \Leftrightarrow H_{ij} \in \AAA\] (where $H_{ji} := H_{ij}$). Letting $C_1, ..., C_r$ be the resulting equivalence classes, we have
\begin{equation}
\sum_{i \, \in \, C_1 \, \sqcup \, \cdots \, \sqcup \, C_r} \partial_i \in \Der(\AAA)
\end{equation}
since this polynomial derivation vanishes on every $\alpha \in \widetilde{\Phi}^+$ for which $H_\alpha \in \AAA$.
We conclude that $\ii_\AAA = S$.

For  (3), suppose first that $\AAA$ is nonessential.
By  (2) we have $\ii_\AAA = S$ and by Lemma~\ref{lem:cospan_containment} we have
$\beta_\AAA \in (S^{\symm_n}_+)$ so that $(S^{\symm_n}_+) : \beta_\AAA = S$, as well.
(This paragraph did not use freeness.)

Now suppose that $\AAA$ is essential and free.
Since $\bigcap_{H \in \AAA} H = \{ 0 \}$, the derivation module $\Der(\AAA)$ has no nonzero elements of degree 0,
which implies that $\ii_\AAA$ has no nonzero elements of degree 0.
Since 
$\beta_\AAA \notin (S^{\symm_n}_+)$ by Lemma~\ref{lem:cospan_containment} and $\ii(\Der(\AAA_{\widetilde{\Phi}^+})) = (S^{\symm_n}_+)$ by \eqref{invariantii},
Lemma~\ref{lem:derivation_colon} applies to prove
 (3).
\end{proof}


For a subset $J \subseteq [n]$, recall that $f_J = \prod_{j \in J} x_j \left( \prod_{j < i \leq n} (x_j - x_i) \right)$.
Proposition \ref{prop:st-general} (3) leads us to consider the arrangement $\AAA_J$ in $\KK^n$ with $\beta_{\AAA} =f_J$. The hyperplanes of $\AAA_J$ are therefore
\begin{equation}
\AAA_J = \{ H_{x_j - x_i} \,:\, j \not \in J, \, i > j \} \cup \{ H_{x_j} \,:\, j \not \in J \}
\end{equation}
so that $\AAA_J \subseteq \AAA_{\widetilde{\Phi}^+}$. 
When $n = 5$ and $J = \{2,4\}$ the hyperplanes of $\AAA_J$ are as follows.

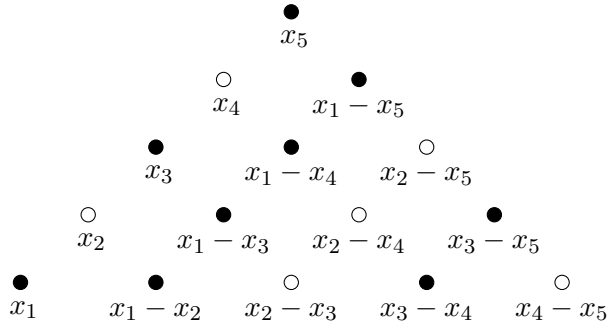
\vspace{0.1in}

\begin{center}
\begin{tikzpicture}[scale = 0.9]

\filldraw (0,0) circle (3pt);
\node at (0.05,-0.4) {$x_1$};

\filldraw (2,0) circle (3pt);
\node at (2,-0.4) {$x_1 - x_2$};

\draw (4,0) circle (3pt);
\node at (4,-0.4) {$x_2 - x_3$};

\filldraw (6,0) circle (3pt);
\node at (6,-0.4) {$x_3 - x_4$};

\draw (8,0) circle (3pt);
\node at (8,-0.4) {$x_4 - x_5$};

\draw (1,1) circle (3pt);
\node at (1.05,0.6) {$x_2$};

\filldraw (3,1) circle (3pt);
\node at (3,0.6) {$x_1 - x_3$};

\draw (5,1) circle (3pt);
\node at (5,0.6) {$x_2 - x_4$};

\filldraw (7,1) circle (3pt);
\node at (7,0.6) {$x_3 - x_5$};

\filldraw (2,2) circle (3pt);
\node at (2.05,1.6) {$x_3$};

\filldraw (4,2) circle (3pt);
\node at (4,1.6) {$x_1 - x_4$};

\draw (6,2) circle (3pt);
\node at (6,1.6) {$x_2 - x_5$};

\draw (3,3) circle (3pt);
\node at (3.05,2.6) {$x_4$};

\filldraw (5,3) circle (3pt);
\node at (5,2.6) {$x_1 - x_5$};

\filldraw (4,4) circle (3pt);
\node at (4.05,3.6) {$x_5$};

\end{tikzpicture}
\end{center}

In particular, we see that the arrangements $\AAA_J$ are not usually southwest, and vice versa.
Like southwest arrangements, the $\AAA_J$ have nice combinatorial properties. Recall the notion of $\stair(J)_i$ from (\ref{eqn:staircase}).

\begin{lemma}
\label{lem:aj-characteristic}
The characteristic polynomial of $\AAA_J$ is given by
$$
\chi_{\AAA_J}(t) = \prod_{i \, = \, 1}^n (t - \stair(J)_i).
$$
In particular, $\AAA_J$ is supersolvable and free with exponents
$(\stair(J)_1, \dots, \stair(J)_n)$.\end{lemma}

\begin{proof}
We apply the Finite Fields Method; see e.g. \cite[Sec. 5.1]{Stanley}.  Let $p$ be a large prime. We have
\begin{equation}
\chi_{\AAA_J}(p) = \# \left(  (\ZZ/p\ZZ)^n - \bigcup_{H \, \in \, \AAA_J} H \right)
\end{equation}
where we interpret the hyperplanes of $\AAA_J$ as living in $(\ZZ/p\ZZ)^n$ by reduction modulo $p$.
We are reduced to showing
\begin{equation}
\# \left(  (\ZZ/p\ZZ)^n - \bigcup_{H \, \in \, \AAA_J} H \right) =  \prod_{i \, = \, 1}^n (p - \stair(J)_i).
\end{equation}
To see this, we consider how a typical point $(a_1, \dots, a_n) \in (\ZZ/p\ZZ)^n$ which avoids the hyperplanes
in $\AAA_J$ is built. For any $j \not \in J$ we have $a_j \neq 0$. Furthermore, for all $1 \leq i \leq n$ the coordinate $a_i$ must be distinct
from all of the coordinates $a_j$ for $j < i$ with $j \notin J.$ The claimed product formula follows.
\end{proof}

We note that
the supersolvablity of $\AAA_J$ can also be checked by showing that the graph $\Gamma(J)$ on $\{0,1,\dots,n\}$ corresponding to $\AAA_J$ is chordal.

Combining Proposition \ref{prop:st-general} (3) and Lemma \ref{lem:aj-characteristic}, we get the following presentation of the algebra $\ST(\AAA_J,\ii)$.

\begin{theorem}
For any $J \subset [n]$,
we have $\ST(\AAA_J,\ii)=
S/((S_+^{\mathfrak S_n}):f_J)$.
\end{theorem}


In the rest of this section,
as a quick application of the above theorem, we give an explicit generating set of the ideal $(S^{\symm_n}_+) : f_J$. 
Rhoades and Wilson showed \cite[Thm.\ 4.10]{RW}
that these ideals are the unit ideal when $1 \in J$, and cut out by a certain regular sequence 
$p_{J,1}, \dots, p_{J,n}$ when $1 \notin J$.
The $p_{J,i}$ appearing in \cite{RW} have a complicated and unusual definition involving partial derivatives
of complete homogeneous symmetric polynomials in partial variable sets.
We will give a simpler regular sequence generating $(S^{\symm_n}_+) : f_J$ by giving a basis of the derivation module of $\AAA_J$.

\begin{proposition}
\label{prop:aj-basis}
A free basis of $\Der(\AAA_J)$ is given by $\rho^J_1, \dots, \rho^J_n$ where 
\begin{equation}
\rho^J_i = 
\begin{cases}
\sum_{k \, = \, i}^n \left(  \prod_{\substack{j \, \not \in \, J \\ j \, < \, i}} (x_j - x_k) \right) x_k \cdot \partial_k &  i \not \in J, \\
\prod_{\substack{j \, \not \in \, J \\ j \, < \, i}} (x_j - x_i) \cdot \partial_i & i \in J.
\end{cases}
\end{equation}
\end{proposition}

\begin{proof}
We apply Saito's Criterion. Adopting the abbreviation $\rho_i := \rho^J_i$,
observe that $\rho_i$ has degree $\stair(J)_i$, so that $\deg \rho_1 + \cdots +\deg \rho_n = |\AAA_J|$
by Lemma~\ref{lem:aj-characteristic}.
Since $\rho_i$ is an $S$-linear combination of $\partial_i, \partial_{i+1}, \dots, \partial_n$ where the coefficient
of $\partial_i$ is nonzero, we see that $\{ \rho_1, \dots, \rho_n \}$ is linearly independent over $S$.

We are reduced to showing that $\rho_i \in \Der(\AAA_J)$ for all $i$. To do this, we argue as follows.

{\bf Case 1:} $i \notin J$.

For any hyperplane $H_\alpha$ of $\AAA_J$, we need to show that $\alpha \mid \rho_i(\alpha)$ in $S$.
For $j \notin J$, we have 
\begin{equation}
\rho_i( x_j ) = \begin{cases}
\prod_{\substack{j' \, \notin \, J \\ j' \, < \, i}} (x_{j'} - x_j) x_j & i \leq j \\
0 & i > j.
\end{cases}
\end{equation}
In either case we have $x_j \mid \rho_i(x_j)$. The remaining hyperplanes of $\AAA_J$ have the form $x_j - x_\ell = 0$ for $j \not \in J$ and $j < \ell$.
Applying $\rho_i$ to the linear form $x_j - x_\ell$ yields
\begin{equation}
\rho_i (x_j - x_\ell) = \begin{cases}
\prod_{\substack{j' \, \notin \, J \\ j' \, < \, i}} (x_{j'} - x_j) x_j -
\prod_{\substack{j' \, \notin \, J \\ j' \, < \, i}} (x_{j'} - x_\ell) x_\ell & i \leq j < \ell, \\
-\prod_{\substack{j' \, \notin \, J \\ j' \, < \, i}} (x_{j'} - x_\ell) x_\ell & j < i \leq \ell, \\
0 & j < \ell < i.
\end{cases}
\end{equation}
In the first branch, setting $x_j = x_\ell$ makes $\rho_i(x_j - x_\ell)$ vanish so that $(x_ j - x_\ell) \mid \rho_i(x_j - x_\ell)$.
When $j < i \leq \ell$ as in the second branch, the factor $(x_j - x_\ell)$ appears in the product so that 
$(x_ j - x_\ell) \mid \rho_i(x_j - x_\ell)$.  The third branch yields $(x_ j - x_\ell) \mid \rho_i(x_j - x_\ell)$ when $j <\ell < i$, as well.
This implies that $\rho_i \in \Der(\AAA_J)$.

{\bf Case 2:} $i \in J$.

If $j \notin J$, applying $\rho_i$ to the linear form $x_j$ yields 0 because $\partial_i(x_j) = 0$.
If $j \notin J$ and $j < \ell$, applying $\rho_i$ to $x_j - x_\ell$ yields
\begin{equation}
\rho_i(x_j - x_\ell) = \begin{cases}
\prod_{\substack{j' \, \notin \, J \\ j' \, < \, i}} (x_{j'} - x_i) & \ell = i, \\
0 & \ell \neq i.
\end{cases}
\end{equation}
Since $x_j - x_i = x_j - x_\ell$ appears in the product in the first branch, we obtain the desired divisibility
$(x_j - x_\ell) \mid \rho_i(x_j - x_\ell)$ and conclude that $\rho_i \in \Der(\AAA_J)$.
\end{proof}

For $J \subseteq [n]$ and $1 \leq i \leq n$, define a polynomial $g_{J,i} \in S$ by 
\begin{equation}
g_{J,i} := \ii(\rho^J_i)=\begin{cases}
\sum_{k \, = \, i}^n \left( \prod_{\substack{j \, \notin \, J, \\ j \, < \, i}} (x_j - x_k) \right) x_k & i \notin J, \\
 \prod_{\substack{j \, \notin \, J \\ j \, < \, i}} (x_j - x_i) & i \in J.
\end{cases}
\end{equation}

\begin{corollary}
\label{cor:generating-set}
Let $J \subseteq [n]$.  If $1 \in J$, the ideal $(S^{\symm_n}_+) : f_J$ is the unit ideal.
If $1 \notin J$, the ideal $(S^{\symm_n}_+) : f_J$ is generated by the regular sequence 
$g_{J,1}, \dots, g_{J,n}$.
\end{corollary}

\begin{proof}
It was established in \cite{RW} that $(S^{\symm_n}_+) : f_J$ is the unit ideal if $1 \in J$.
(This also follows from Lemma \ref{lem:cospan_containment}.)
Suppose $1 \not \in J$.
Then $\rho_1^{J},\dots,\rho_n^{J}$ have positive degrees.
Up to a nonzero scalar we have
\begin{equation}
f_J = Q(\AAA_{\widetilde{\Phi}^+}) / Q(\AAA_{J}).
\end{equation}
By Proposition~\ref{prop:aj-basis} and  
Proposition~\ref{prop:st-general} (3) we see that $\ii_{\AAA_{J}}=(S^{\symm_n}_+) : f_J$ is generated by the regular 
sequence $\ii(\rho^J_1), \dots, \ii(\rho^J_n) \in S$.
\end{proof}

\section{Monomial bases of Solomon--Terao algebras}
\label{Monomial}

Our goal is to show that the algebra $S/((S^{\mathfrak S_n}_+):f_J)=\ST(\AAA_J,\ii)$ has the expected monomial basis $\mathcal M(J)$ given in the introduction. To do this, we first give a basis of Solomon--Terao algebras $\ST(\AAA,\ii)$ for southwest arrangements $\AAA$.
The key lemma to do this is the following short exact sequence of Solomon--Terao algebras $\ST(-, \ii)$ induced by the short exact sequence of derivation modules in Theorem~\ref{thm:exact-triple}.


\begin{lemma}
\label{lem:st-short-exact}
Let $\AAA$ be a subarrangement of $\AAA_{\widetilde{\Phi}^+}$ and $H = H_\alpha \in \AAA$.
Assume that the triple $(\AAA \setminus H, \AAA, \AAA^H)$ satisfies the conditions of 
Theorem~\ref{thm:addition}.
\begin{enumerate}
\item We have $\ST(\AAA^H,\ii) = S/(\ii_\AAA + (\alpha))$.
\item We have a short exact sequence
$$
0 \longrightarrow \ST(\AAA \setminus H,\ii) \xrightarrow{ \, \, \times \alpha \, \, } \ST(\AAA,\ii) \longrightarrow \ST(\AAA^H,\ii) \longrightarrow 0
$$
where the second map is the  canonical projection.
\end{enumerate}
\end{lemma}

\begin{proof}
Let $\eta_H: \Der(\AAA) \twoheadrightarrow \Der(\AAA^H)$ be the surjection appearing in Theorem~\ref{thm:exact-triple}.
Write $S' := S/( \alpha )$ and let $\pi: S \twoheadrightarrow S'$ be the canonical surjection.
We have a commutative diagram
\begin{center}
\begin{tikzpicture}[scale = 1]

\node (A) at (0,0) {$S$};

\node (B) at (4,0) {$S'$};

\node (C) at (0,2) {$\Der(\AAA)$};

\node (D) at (4,2) {$\Der(\AAA^H)$};

\draw [->] (A) -- (B) node[midway, below] {$\pi$};

\draw [->] (C) -- (D) node[midway, above] {$\eta_H$};

\draw [->] (C) -- (A) node[midway, left] {$\ii$};

\draw [->] (D) -- (B) node[midway, right]{$\ii$};

\end{tikzpicture}
\end{center}
where the horizontal maps are surjections. We conclude that
\begin{equation}
\ii_{\AAA^H} = \ii(\Der(\AAA^H)) = (\pi \circ \ii) (\Der(\AAA)) = (\ii_\AAA + ( \alpha )) / (\alpha)
\end{equation}
so that
\begin{equation}
\ST(\AAA^H,\ii) = S'/\ii_{\AAA^H} = S / (\ii_\AAA + (\alpha)).
\end{equation}
This proves (1).

For (2), the definition of colon ideals gives the short exact sequence
\begin{equation}
0 \longrightarrow S/ (\ii_\AAA : (\alpha)) \xrightarrow{ \, \, \times \alpha \, \, } S / \ii_\AAA \longrightarrow S / (\ii_\AAA + (\alpha)) \longrightarrow 0
\end{equation}
where the second map is the canonical projection.
The middle term in this short exact sequence is $\ST(\AAA,\ii)$.
By (1), the right term is $\ST(\AAA^H,\ii)$.
The left term is $\ST(\AAA \setminus H, \ii)$ by Lemma~\ref{lem:derivation_colon}.
\end{proof}

Let $\AAA \subseteq \AAA_{\widetilde{\Phi}^+}$ be an essential southwest arrangement 
with $h$-sequence $h(\AAA) = (h_1, \dots, h_n)$.
Since $\ii$ is homogeneous of degree 0, 
Lemma~\ref{lem:st-hilbert-series} and
Proposition~\ref{prop:derivation-module-basis} imply that $\ST(\AAA,\ii)$ has Hilbert series
\begin{equation}
\Hilb( \ST(\AAA,\ii); q) = \prod_{i \, = \, 1}^n \frac{1 - q^{h_i}}{1-q}.
\end{equation}
This says that the set 
$ \{ x_1^{a_1} \cdots x_n^{a_n} \,:\, a_i < h_i \text{ for all $i$} \}$ of monomials in $S$
have the correct multiset of degrees to descend to a basis of $\ST(\AAA,\ii)$.
The main theorem of this section establishes this fact using short exact sequences.
The authors are unaware of a straightening or Gr\"obner-type proof of this result.

\begin{theorem}
\label{thm:sw-basis}
Let $\AAA$ be an essential southwest arrangement in $\KK^n$ with
$h(\AAA) = (h_1, \dots, h_n)$. The set 
\[ \{ x_1^{a_1} \cdots x_n^{a_n} \,:\, a_i < h_i \text{ for all $i$} \} \]
descends to a $\KK$-basis of $\ST(\AAA,\ii)$.
\end{theorem}

\begin{proof}
Since $\AAA$ is essential we have $|\AAA| \geq n$.  We use double induction on $n$ and $|\AAA|$.
If $n = 1$ the result is an easy computation.
If $|\AAA| = n$, the ideal $\ii_\AAA$ is generated by $n$ linearly independent homogeneous polynomials of degree 1
so that $\ST(\AAA,\ii) = S/(x_1, \dots, x_n)$ and $h(\AAA) = (1, \dots, 1)$ so the result follows.

Now assume $n > 1$ and $|\AAA| > n$.  Since $\AAA$ is an essential southwest arrangement, 
it contains at least one coordinate hyperplane.
Let $p := \max \{ k \,:\, H_k \in \AAA \}$.
Lemma~\ref{lem:sw-addition} and Lemma~\ref{lem:st-short-exact} (2) give a short exact sequence
\begin{equation}
\label{eq:key-short-exact-st}
0 \longrightarrow \ST(\AAA \setminus H_p,\ii) \xrightarrow{ \, \, \times x_p \, \, } \ST(\AAA,\ii) \longrightarrow \ST(\AAA^{(p)},\ii) \longrightarrow 0
\end{equation}
where $\AAA^{(p)} := \AAA^{H_p}$.
We inductively assume that the theorem holds for the arrangements $\AAA \setminus H_p$ and $\AAA^{(p)},$ which are themselves southwest arrangements by Lemma~\ref{lem:sw-recursion}.
Since $H_p$ contributes to $h_p$, we have $h_p \geq 1$. Our argument splits into cases depending on the value of $h_p$.

{\bf Case 1:} {\em We have $h_p > 1$.}

Lemma~\ref{lem:sw-recursion} states that 
\begin{center}
$h(\AAA \setminus H_p) = (h_1, \dots, h_{p-1}, h_p - 1, h_{p+1}, \dots, h_n)$ and 
$h(\AAA^{(p)}) = (h_1, \dots, h_{p-1}, h_{p+1}, \dots, h_n)$.
\end{center}
The result for $\ST(\AAA,\ii)$ follows from the form of the short exact sequence \eqref{eq:key-short-exact-st} and induction.

{\bf Case 2:} {\em We have $h_p = 1$.}

In this case the deleted arrangement $\AAA \setminus H_p$ is not essential.
Proposition~\ref{prop:st-general} (2) implies that $\ST(\AAA \setminus H_p,\ii) = 0$, so the map
$ \ST(\AAA,\ii) \xrightarrow{ \, \, \sim \, \, } \ST(\AAA^{(p)},\ii)$ appearing in \eqref{eq:key-short-exact-st} induced by setting $x_p = 0$ is an isomorphism.
The result for $\ST(\AAA,\ii)$ follows from induction.
\end{proof}

\begin{remark}
Theorem \ref{thm:sw-basis} may be considered as a generalization of the main result in \cite{HHMPT}, which gave a monomial basis for the cohomology ring of regular nilpotent Hessenberg varieties in type $A$. Indeed, these cohomology rings are known to be equal to Solomon--Terao algebras $\ST(\AAA,\aaa)$ of type $A$ ideal arrangements $\AAA$ (that is, subarrangements of $\AAA_{\Phi^+}$ which are both southwest-closed and southeast-closed).
But by setting $\widetilde \AAA=\AAA \cup\{H_{x_1},\dots,H_{x_n}\}$,
this algebra also equals $\ST(\widetilde \AAA,\ii)$ by Proposition \ref{prop:st-general} (3),
so Theorem \ref{thm:sw-basis} also gives a monomial basis for those rings.
\end{remark}

\section{Artin basis of $SR$}
\label{Artin}

In this section we apply
Theorem~\ref{thm:sw-basis} to the superspace coinvariant ring $SR$.
For $J \subseteq [n]$, recall that $\MMM(J)$ denotes the following set of monomials in $S$:
\[ \MMM(J) := \{ x_1^{a_1} \cdots x_n^{a_n} \,:\, a_i < \stair(J)_i \text{ for all $i$} \} \]
and $\MMM$ is the set of superspace monomials in $\Omega$:
\[ \MMM := \bigsqcup_{J \, \subseteq \, [n]} \MMM(J) \cdot \theta_J. \]
The following result proves a conjecture of Rhoades and Wilson \cite[Conj. 5.5]{RW}.

\begin{theorem}
\label{thm:ss-J}
For any subset $J \subseteq [n]$, the family $\MMM(J)$ descends to a basis of $S/( (S^{\symm_n}_+) : f_J)$.
\end{theorem}

\begin{proof}
Rhoades and Wilson showed \cite{RW} that the cardinality of $\MMM(J)$ coincides with the dimension of the vector space 
$S/( (S^{\symm_n}_+) : f_J)$, so it is enough to show that $\MMM(J)$ is $\KK$-linearly independent in $S/( (S^{\symm_n}_+) : f_J)$.
(This also follows from Lemma \ref{lem:st-hilbert-series} (3) and Corollary~\ref{cor:generating-set}.)

We introduce the modified polynomial $\widetilde{f}_J \in S$ given by
\begin{equation}
\widetilde{f}_J := \prod_{j \, \in \, J} \left(  \prod_{i \, = \, j+1}^n (x_j - x_i)  \right).
\end{equation}
That is, the polynomial $\widetilde{f}_J$ is the same as the polynomial $f_J$, but with the $x_j$ factors for $j \in J$  removed. Let $\AAA \subseteq \AAA_{\widetilde{\Phi}^+}$
be the subarrangement
\begin{equation}
\AAA := \{ H_{j,i} \,:\, j \notin J, \, i = j+1, \dots, n \} \cup \{ H_{x_1}, H_{x_2}, \dots, H_{x_n} \}=\AAA_J \cup\{H_{x_k} \,:\, k \in J\}.
\end{equation}
Observe that $\AAA$ is a southwest arrangement. Furthermore, we have an injection
\begin{equation}
\label{eq:important-injection}
S/ ((S^{\symm_n}_+) : f_J) \xrightarrow{ \, \, \times \prod_{j \in J} x_j \, \, } S / ((S^{\symm_n}_+) : \widetilde{f}_J) = \ST(\AAA,\ii)
\end{equation}
where the equality $S / ((S^{\symm_n}_+) : \widetilde{f}_J) = \ST(\AAA,\ii)$
is a consequence of Proposition~\ref{prop:sw-free} and Proposition~\ref{prop:st-general} (3).
If $h(\AAA) = (h_1, \dots, h_n)$, it can be checked that
\begin{equation}
h_k = \begin{cases}
\stair(J)_k & k \notin J, \\
\stair(J)_k + 1 & k \in J.
\end{cases}
\end{equation}
Theorem~\ref{thm:sw-basis} says that the set of monomials
\begin{equation}
\left\{
x_1^{a_1} \cdots x_n^{a_n} \,:\,
\begin{array}{c}
a_k < \stair(J)_k \text{ if $k \notin J$,} \\
a_k < \stair(J)_k + 1 \text{ if $k \in J$}
\end{array}
\right\}
\end{equation}
descends to a basis of $S / ((S^{\symm_n}_+) : \widetilde{f}_J)$.
The injection \eqref{eq:important-injection} applies to show that 
\begin{equation}
\MMM(J) = \{ x_1^{a_1} \cdots x_n^{a_n} \,:\, a_k < \stair(J)_k \text{ for all $k$} \}
\end{equation}
is linearly independent in 
$S/( (S^{\symm_n}_+) : f_J)$.
\end{proof}

The Sagan--Swanson Conjecture~\ref{conj:ss} is a consequence of Theorem~\ref{thm:rw} 
and Theorem~\ref{thm:ss-J}.

\begin{corollary}
\label{cor:ss}
The set $\MMM$ of superspace Artin monomials descends to a basis of $SR$.
\end{corollary}

In his Ph.D. thesis at York University, Kelvin Chan conjectured \cite{Chan} a (non-monomial) bihomogeneous basis of $SR$ which respects the $\symm_n$-isotypic decomposition. 
Proving that Chan's conjectural basis is a basis would imply the bigraded $\symm_n$-structure of $SR$ predicted by the Fields Group. 
However, since the supercommutative-to-commutative transfer principle in Theorem~\ref{thm:rw} is non-equivariant, new ideas will be required to prove Chan's conjecture.

It may also be interesting to find an analogue of Steinberg's {\em descent basis} \cite{SteinbergP} for the coinvariant ring $S/(S^{\symm_n}_+)$ for the superspace coinvariants $SR$. Such an object could aid in understanding the 0-Hecke module structure of $SR$; see \cite{HR}. The descent basis consists of monomials, so a transfer principle like Theorem~\ref{thm:rw} might be more applicable here.

\section{Conclusion}
\label{Conclusion}

The Solomon--Terao algebras $\ST(\AAA,\ccc)$ 
have seen two major applications.
When $\AAA = \AAA_\III$ is an ideal subarrangement of a Weyl arrangement and $\aaa: \partial_i \mapsto x_i$,
the algebra $\ST(\AAA,\aaa)$ presents the cohomology of a regular nilpotent Hessenberg variety \cite{AHMMS, AMMN, HHMPT}.
When $\AAA$ is a southwest subarrangement of the augmented Weyl arrangement $\AAA_{\widetilde{\Phi}^+}$ of type A or $\AAA = \AAA_J$ and
$\ii: \partial_i \mapsto 1$, we established connections between $\ST(\AAA,\ii)$ and the type A superspace coinvariant ring.
This motivates the following question.

\begin{question}
\label{q:other-types}
Do the augmented root system $\widetilde{\Phi}^+$, the southwest arrangements, or the $\AAA_J$ arrangements extend to types other than A?
\end{question}

An answer to Question~\ref{q:other-types} should be a class of free arrangements whose algebraic properties are tied 
to the invariant degrees of the Weyl group in question.
Ideally, an answer to Question~\ref{q:other-types} could be applied to
superspace coinvariant rings of other types.
If $W$ is a complex reflection group acting on 
$V = \CC^n$, the natural action 
of $W$ on $V$ induces an 
action on the space $\Omega$ of polynomial-valued
differential forms on $V$. Solomon \cite{Solomon}
gave a uniform description of the
structure of the invariant ring $\Omega^W$.
The quotient $SR_W := \Omega/(\Omega^W_+)$ is a 
bigraded $W$-module, where $(\Omega_+^W) \subseteq \Omega$ is the ideal generated by $W$-invariants with vanishing constant term.
Sagan and Swanson conjectured \cite[Conj. 6.17]{SS}
a basis for $SR_W$ when $W = \symm_{\pm n}$ is the 
group of $n \times n$ signed permutation matrices.
Proving their conjecture could involve establishing a signed permutation analog of Theorem~\ref{thm:rw} so that bases of $SR_W$ could be found using commutative algebra alone.

In type $W = \symm_n$, the dimension of $SR_W = SR$ is given by the number of faces in the {\em Coxeter complex} attached to $W$, i.e. the number of ordered 
set partitions of $[n]$. 
The singly-graded fermionic Hilbert series is the 
reversal of the $f$-vector of the Coxeter complex
(counting ordered set partitions by the number 
of their blocks).
Sagan and Swanson conjecture \cite{SS} that the same is true for $W = \symm_{\pm n}$ has type B$_n$. However, in type F$_4$, the fermionic Hilbert series
is $1152 q^0 + 2304 q^1 + 1396 q^2 + 244 q^3 + q^4$
whereas the reversed $f$-vector is 
$(1152, 2304, 1392, 240,1)$.
Can the theory of arrangements explain this near-miss?

\begin{question}
\label{q:geometry}
Let $\AAA$ be a southwest arrangement. Is there a cohomological interpretation of $\ST(\AAA,\ii)$?
What about $\ST(\AAA_J,\ii)$ for $J \subseteq [n]$?
\end{question}

Proposition~\ref{prop:st-general} gives an interpretation of the Solomon--Terao algebras in Question~\ref{q:geometry} in terms of colon ideal quotients.
In \cite{AHMMS} the cohomology rings of regular nilpotent Hessenberg varieties are presented as quotients by colon ideals;
the methods of \cite{AHMMS} may help in addressing Question~\ref{q:geometry}.
We remark that the polynomials $f_J$ bear a resemblance to ideal generators for a Borel-type presentation of the cohomology of regular nilpotent 
Hessenberg varieties in type A due to 
Abe-Harada-Horiguchi-Masuda \cite{AHHM} and 
reinerpreted and generalized using arrangements by 
Abe-Horiguchi-Masuda-Murai-Sato \cite{AHMMS}.

\begin{question}
\label{q:st-general}
Do other Solomon--Terao algebras $\ST(\AAA,\ccc)$ 
have applications outside of hyperplane arrangements?
\end{question}

When $\AAA$ is free, the Solomon--Terao algebra $\ST(\AAA,\ccc)$ is a Poincar\'e duality algebra as soon as it is of positive finite $\KK$-dimension,
so Question~\ref{q:st-general} would appear well suited to geometric answers.
On the other hand, when $\AAA$ is not free, it can be difficult to understand the structure of $\Der(\AAA)$, and the algebra $\ST(\AAA,\ccc)$ 
can be correspondingly mysterious.
It would be interesting to calculate and find applications for Solomon--Terao algebras outside the bounds of freeness.

\section*{Acknowledgements}

This project was initiated at the Minnesota Research Workshop in Algebra and Combinatorics (MRWAC) at the University of Minnesota in May 2023.
We thank Elise Catania, Sasha Pevzner, and Sylvester Zhang for organizing this workshop.
Further discussions took place at the Hokkaido Summer Institute at the University of Hokkaido and in the Tokyo Metropolitan Government Building in Summer 2023.
We thank the University of Minnesota, the University of Hokkaido, and the Tokyo Metropolitan Government for providing excellent working conditions.
The authors are grateful to Vic Reiner for helpful conversations. 
The authors are grateful to Sylvester Zhang for very fruitful discussions and helpful assistance throughout this project.
The authors thank two anonymous referees for their comments which substantially improved the exposition.
P. Commins was partially supported by the NSF GRFP Award \# 2237827.
S. Murai was partially supported by KAKENHI 21K0319, 21H00975 and 23K17298.
B. Rhoades was partially supported by NSF Grant DMS-2246846.


\section{Appendix: Proof of Lemma~\ref{lem:cospan_containment}}
\label{sec:appendix}

Throughout this appendix, write $V = \KK^n$ and $V^* = \KK \cdot \{x_1, \dots, x_n\}$ where $x_i$ is the $i^{th}$ standard coordinate function on $V$. Before proving Lemma~\ref{lem:cospan_containment}, we recall some useful facts about the classical coinvariant ideal $(S^{\symm_n}_+)$.

\subsection{Symmetric polynomials and $(S^{\symm_n}_+)$}
We need standard terminology related to symmetric polynomials.
For $n \geq 0$, a {\em partition} $\lambda$ of $n$ is a weakly decreasing sequence $\lambda = (\lambda_1 \geq \lambda_2 \geq \cdots )$ 
which satisfies $\lambda_1 + \lambda_2 + \cdots = n$.  We write $\lambda \vdash n$ to mean that $\lambda$ is a partition of $n$,
and allow trailing zeros at the end of a partition, so that $(4,2,2,1) = (4,2,2,1,0,0) \vdash 9$.
We write $\ell(\lambda)$ for the number of nonzero parts of $\lambda$.

For $d > 0$, the {\em elementary symmetric polynomial} of degree $d$ is the polynomial $e_d \in S$ given by the 
sum of all squarefree monomials in $S$ of degree $d$:
\begin{equation}
e_d := \sum_{1 \leq i_1 < \cdots < i_d \leq n} x_{i_1} \cdots x_{i_d}.
\end{equation}
(These are not to be confused with the exponents of a free arrangement.)
The {\em complete homogeneous symmetric polynomial} of degree $d$ is given by the sum of all monomials in $S$ of degree $d$:
\begin{equation}
h_d := \sum_{1 \leq i_1 \leq \cdots \leq i_d \leq n} x_{i_1} \cdots x_{i_d}.
\end{equation}
By convention, we let $e_0 = h_0 = 1$ and $e_d = h_d = 0$ when $d < 0$.
If $\lambda = (\lambda_1, \lambda_2, \dots )$ is a partition, the {\em Schur polynomial} $s_\lambda \in S$ may be defined by the 
{\em Jacobi-Trudi formula}
\begin{equation}
s_\lambda =  \det \left( h_{\lambda_i-i+j} \right)_{1 \leq i,j \leq \ell(\lambda)}.
\end{equation}

If $A \subseteq [n]$ and $f \in S^{\symm_n}$ is a symmetric polynomial, write $f(A)$ for the polynomial $f$ in the variable set $\{ x_a \,:\, a \in A \}$ 
obtained by setting all variables with indices outside of $A$ to zero.
We write $\delta_n \in S$ for the classical Vandermonde determinant
\begin{equation}
\delta_n := \prod_{1 \, \leq i \, < \, j \, \leq \, n} (x_i - x_j) = 
\sum_{w \, \in \, \symm_n} \sign(w) \cdot x_{w(1)}^{n-1} x_{w(2)}^{n-2} \cdots x_{w(n)}^0.
\end{equation}

\begin{lemma}
    \label{lem:coinvariant-tricks}
    Recall that $(-) \odot (-): S \times S \rightarrow S$ is the action of $S$ on itself by partial differentiation.
    \begin{enumerate}
        \item Given $f \in S$, we have $f \in (S^{\symm_n}_+)$ if and only if $f \odot \delta_n = 0$.
        \item If $A \subseteq [n]$ and a partition $\lambda$ satisfies  $\lambda_1 > n-|A|$, we have  $s_\lambda(A) \in (S^{\symm_n}_+)$.
        \item If $[n] = A \sqcup B$   we have $h_d(A) \equiv (-1)^d e_d(B)$  for each $d \geq 0$ where the congruence is modulo $(S^{\symm_n}_+)$.
    \end{enumerate}
\end{lemma}

\begin{proof}
    (1) is a well-known result of Steinberg \cite{Steinberg}.

    For (2), suppose first that $\lambda$ has only one part, so that
    $s_{(\lambda_1)} = h_{\lambda_1}$. We have the identities
    \begin{equation}
    \label{eq:h-identity}
    h_d(A) =h_d(A \cup \{a\}) - x_a h_{d-1}(A \cup \{a\}) 
    \end{equation}
    for all $a \not \in A$ and the membership 
    \begin{equation}
    \label{eq:h-membership}
    \text{$h_d  = h_d([n]) \in (S^{\symm_n}_+)$ for all $d > 0$.}
    \end{equation} 
    Using the identity \eqref{eq:h-identity} and the membership \eqref{eq:h-membership}, (2) follows by downward induction on $n - |A|$.
   
    We now prove (2) for general partitions $\lambda$.
    The Jacobi-Trudi identity says
    that
    \[
    s_\lambda(A) = \det \left( h_{\lambda_i-i+j}(A) \right)_{1 \leq i,j \leq \ell(\lambda)}
    \]
    and the conclusion of (2) follows by expanding
    this determinant along the first row, applying that $s_{(\lambda_1)}(A) \in  (S^{\symm_n}_+)$.
    
    For (3), write $R = S/(S^{\symm_n}_+)$ for the classical coinvariant ring. We introduce a new variable $t$ and work in the ring $R[t]$. In the ring $R[t]$ there holds the identity
    \begin{equation}
    \label{eq:coinvariant-tricks-one}
    \prod_{i \, = \, 1}^n (1 - x_i t) = 1
    \end{equation}
    since the coefficient of $t^d$ on the left hand side of \eqref{eq:coinvariant-tricks-one} is $(-1)^d e_d$. For any $1 \leq i \leq n$, the special case $A = \{i\}$ and $\lambda = (n)$ of (2) says that $x_i^n \in (S^{\symm_n}_+)$, so we have $x_i^n = 0$ in $R$. In particular, the expression 
    \[
    \frac{1}{1 - x_i t} := 1 + x_i t + x_i^2 t^2 + x_i^3 t^3 + \cdots 
    \]
    gives a well defined element of $R[t]$. It is not hard to see that $(1 - x_i t)$ and $\frac{1}{1 - x_i t}$ are mutually inverse in $R[t]$.
    We multiply both sides of \eqref{eq:coinvariant-tricks-one} by $\prod_{a \, \in \, A} \frac{1}{1- x_a t}$ to obtain the identity
    \begin{equation}
    \label{eq:coinvariant-tricks-two}
    \prod_{b \, \in \, B} (1 - x_b t) =
    \prod_{a \, \in \, A} \frac{1}{1 - x_a t}
    \end{equation}
    in $R[t].$
    Comparing
    the coefficients of $t^d$ on both sides, we see that $(-1)^d e_d(B) = h_d(A)$ in $R$, which is equivalent to (3).
\end{proof}

\subsection{Proving the cospanning lemma}
We have all the tools we need to prove Lemma~\ref{lem:cospan_containment}. For the convenience of the reader, we recall its statement.

\begin{replemma}{lem:cospan_containment}  Let $T \subseteq \widetilde{\Phi}^+$. We have $\prod_{\alpha \in T} \alpha \in (S^{\symm_n}_+)$  if and only if  the complement $\widetilde{\Phi}^+ - T$  of $T$ does not span $V^* = \KK \cdot \{x_1, \dots, x_n\}$.
\end{replemma}

\begin{proof}[Proof of Lemma~\ref{lem:cospan_containment}]
    In this proof we will use the well-known fact that
    \begin{equation}
        \label{eq:schur-vanish}
        s_\lambda(A) = 0 \text{ if } \ell(\lambda) > |A|.
    \end{equation}
The vanishing property \eqref{eq:schur-vanish} follows from the combinatorial interpretation of $s_\lambda$ as the content generating function
of semistandard tableaux of shape $\lambda$. We write
\[
T^c := \widetilde{\Phi}^+ - T
\]
for the complement of $T$ in $\widetilde{\Phi}^+$.

    Suppose first that $T^c$  does not span $V^*$. Define a subset $S \subseteq [n]$ by
    \begin{equation*}
        B := \{1 \leq b \leq n \,:\, x_b \in \mathrm{span}_\KK(T^c) \}.
    \end{equation*}
    Since $T^c$ does not span $V^* = \KK \cdot \{x_1, \dots, x_n\}$ we have $B \neq [n]$. Define 
    \[
    A := [n] - B
    \]
    to be the complement of $B$ within $[n]$. We know that $A \neq \varnothing$ and 
    \begin{center}
    $x_a \in T$ for each $a \in A$
    and $\pm (x_a - x_b) \in T$ for all $a \in A$ and $b \in B$.
    \end{center}
     It suffices to establish the membership
    \begin{equation}
    \label{eq:new-desired-membership}
    \prod_{a \, \in \, A} x_a \times 
    \prod_{\substack{a \, \in \, A \\ b \, \in \, B}} (x_a - x_b) \in 
    (S^{\symm_n}_+)
     \end{equation}
     since the displayed polynomial divides  $\prod_{\alpha \in T} \alpha$.

     We turn to proving the membership \eqref{eq:new-desired-membership}. Expanding the product $\prod_{\substack{a \in A \\ b \in B}} (x_a - x_b)$ yields
    \begin{align}
    \label{eq:containment-one}
        \prod_{a \, \in \, A} x_a \times 
    \prod_{\substack{a \, \in \, A \\ b \, \in \, B}} (x_a - x_b) &=
    \prod_{a \, \in \, A} x_a \times 
    \prod_{a \, \in \, A} \left(
    \sum_{d \, = \, 0}^{|B|} (-1)^d x_a^{|B| - d}
    e_d(B)
    \right) \\ 
\nonumber    &\equiv
    \prod_{a \, \in \, A} x_a \times 
    \prod_{a \, \in \, A} \left(
    \sum_{d \, = \, 0}^{|B|} x_a^{|B| - d}
    h_d(A) \right) \mod (S^{\symm_n}_+)
    \end{align}
    where the congruence  applies Lemma~\ref{lem:coinvariant-tricks} (3). However, we have
    \begin{multline}
        \label{eq:membership-congruence-helper}
        \prod_{a \, \in \, A} x_a \times h_{|B|}(A)
        = e_{|A|}(A) \cdot h_{|B|}(A) = 
        s_{\left(|B|,1^{|A|}\right)}(A) +
        s_{\left(|B|+1,1^{|A|-1}\right)}(A)  \\ =
        s_{\left(|B|+1,1^{|A|-1}\right)}(A)
        \in (S^{\symm_n}_+)
     \end{multline}
where the second equality uses the Pieri rule and the third uses the fact that number of parts of $(|B|,1^{|A|})$ is $|A|+1$ together with the vanishing property \eqref{eq:schur-vanish}. The ideal membership is justified by 
    Lemma~\ref{lem:coinvariant-tricks} (2). Equation~\eqref{eq:containment-one}  therefore continues as 
\begin{align}
    \prod_{a \, \in \, A} x_a \times 
    \prod_{a \, \in \, A} \left(
    \sum_{d \, = \, 0}^{|B|} x_a^{|B| - d}
    h_d(A) \right)  &\equiv \prod_{a \in A} x_a \times \prod_{a \in A} \left( \sum_{d\, =\, 0}^{|B|-1} x_a^{|B|-d} h_d(A) \right)  \\
    &= \left( \prod_{a \, \in \, A} x_a \right)^2 \times 
    \prod_{a \, \in \, A} \left(
    \sum_{d \, = \, 0}^{|B|-1} x_a^{|B| - d - 1}
    h_d(A) \right). \label{eq:containment-three}
\end{align}
The congruence $\equiv$ is modulo $(S^{\symm_n}_+)$ (and is justified by the membership \eqref{eq:membership-congruence-helper}) while the equality pulls out a factor of $x_a$ from each term of the sum in parentheses.
    We have the membership
    \begin{equation}
        \label{eq:containment-four}
        \left( \prod_{a \, \in \, A} x_a \right)^2
        \times 
        h_{|B|-1}(A) = e_{|A|}(A)^2 \cdot 
        h_{|B|-1}(A) = 
        s_{\left(|B|+1,2^{|A|-1}\right)}(A) \in 
        (S^{\symm_n}_+)
    \end{equation}
    where the second equality follows from the Pieri rule and the vanishing property \eqref{eq:schur-vanish}. The ideal membership $s_{\left(|B|+1,2^{|A|-1}\right)}(A) \in 
        (S^{\symm_n}_+)$ is again justified
    by Lemma~\ref{lem:coinvariant-tricks} (2). 
    The expression~\eqref{eq:containment-three} may be rewritten as
   \begin{align}
    \left( \prod_{a \, \in \, A} x_a \right)^2 \times 
    \prod_{a \, \in \, A} \left(
    \sum_{d \, = \, 0}^{|B| - 1} x_a^{|B| - d - 1}
    h_d(A) \right)  &\equiv \left( \prod_{a \in A} x_a \right)^2 \times \prod_{a \in A} \left( \sum_{d\, =\, 0}^{|B|-2} x_a^{|B|-d-1} h_d(A) \right)  \\
    &= \left( \prod_{a \, \in \, A} x_a \right)^3 \times 
    \prod_{a \, \in \, A} \left(
    \sum_{d \, = \, 0}^{|B|-2} x_a^{|B| - d - 2}
    h_d(A) \right) 
\end{align}
    Continuing in this fashion, we eventually see that the quantity in Equation~\eqref{eq:containment-one} satisfies
    \begin{equation}
    \prod_{a \, \in \, A} x_a \times 
    \prod_{a \, \in \, A} \left(
    \sum_{d \, = \, 0}^{|B|} x_a^{|B| - d}
    h_d(A) \right) \equiv
    \left( \prod_{a \, \in \, A} x_a \right)^{|B|+1} = s_{\left((|B|+1)^{|A|}\right)}(A) 
    \equiv 0 \mod (S^{\symm_n}_+).
    \end{equation}
    The equality is the definition of the Schur function indexed by an $|A|$-by-$(|B|+1)$ rectangle and the congruence once again applies Lemma~\ref{lem:coinvariant-tricks} (2) together with the fact that $A \neq \varnothing$. This completes the proof in the case where $T^c$ does not span $V^*$.

    Now assume that $T^c$ does span $V^* = \KK \cdot \{x_1, \dots, x_n\}$. Define a graph $\Gamma$ on the vertex set $[n]$ by letting $i \sim j$ be an edge in $\Gamma$ for all $1 \leq i < j \leq n$ such that $x_i - x_j \in T^c$. Since $T^c$ spans $V^*$, each connected component $C$ of $\Gamma$ contains at least one vertex $1 \leq i \leq n$ such that $x_i \in T^c$. By  selecting a distinguished such vertex from each connected component and removing edges from $\Gamma$ if necessary, we obtain a forest $F$ on the vertex set $[n]$ with trees 
    \[F = \tau_1 \sqcup \dots \sqcup \tau_r\]
    such that
    \begin{itemize}
        \item for all edges $i < j$ we have $x_i - x_j \in T^c$, and
        \item for each $1 \leq p \leq r$  there exists a distinguished vertex $i_p$ in $\tau_p$  such that   $x_{i_p} \in T^c$.
    \end{itemize}
    It follows that $\prod_{\alpha \in T} \alpha$ divides
    \begin{equation}
    \label{eq:containment-six}
    \prod_{p \, = \, 1}^r 
    \prod_{\substack{i \, \in \, \tau_p \\ i \, \neq \, i_p}}
    x_{i} \times
    \prod_{\substack{1 \, \leq i \, < \, j \, \leq \, n \\ i \, \not\sim_F \, j}} (x_i - x_j)
    \end{equation}
    where in the second product $i \not\sim_F j$ means that $i$ and $j$ are not connected by an edge in $F$.

    It suffices to show that  the expression~\eqref{eq:containment-six} does not lie in $(S^{\symm_n}_+)$. Equivalently (by Lemma~\ref{lem:coinvariant-tricks} (1)) we need to prove 
    \begin{equation}
        \label{eq:containment-seven}
        \left(
    \prod_{p \, = \, 1}^r 
    \prod_{\substack{i \, \in \, \tau_p \\ i \, \neq \, i_p}}
    x_{i}  \times
    \prod_{\substack{1 \, \leq \, i \, < \, j \, \leq \, n \\ i \, \not\sim_F \, j}} (x_i - x_j)
    \right) \odot \delta_n \neq 0.
    \end{equation}
    where $(-) \odot (-)$ is the action of $S$ on itself by partial differentiation. Since 
    \begin{equation}
    \label{eq:vandermonde-expansion}
        \delta_n = \sum_{w \, \in \, \symm_n} \sign(w) \cdot x^{n-1}_{w(1)} x^{n-2}_{w(2)} \cdots x^1_{w(n-1)} x^0_{w(n)}
    \end{equation}
    and the polynomial \eqref{eq:containment-six} is homogeneous of degree ${n \choose 2}$, we conclude that
    \begin{quote}
     $(\heartsuit)$ the monomials $m$ in the expansion of \eqref{eq:containment-six} for which $m \odot \delta_n \neq 0$ are precisely those of the form $m = x^{n-1}_{w(1)} x^{n-2}_{w(2)} \cdots x^1_{w(n-1)} x^0_{w(n)}$ for some $w \in \symm_n$.
    \end{quote}

    The expression \eqref{eq:containment-six} is almost the Vandermonde determinant $\delta_n$. In fact, if $t(i)$ is the vertex in $F$ adjacent to $i$ in the unique path from $i$ to $i_p$, we have
    \begin{equation}
        \label{eq:containment-eight}
        \prod_{p \, = \, 1}^r 
    \prod_{\substack{i \, \in \, \tau_p \\ i \, \neq \, i_p}}
    \left(x_i - x_{t(i)}\right) \times
    \prod_{\substack{1 \, \leq \, i \, < \, j \, \leq \, n \\ i \, \not\sim_F \, j}} (x_i - x_j) = \pm \prod_{1 \, \leq \, i \, < \, j \, \leq \, n} (x_i - x_j) =
    \pm \delta_n.
    \end{equation}
    The product in \eqref{eq:containment-eight} has ${n \choose 2}$ binomial factors, so its expansion has $2^{n \choose 2}$ terms. Since  \eqref{eq:vandermonde-expansion} has only $n!$ summands,  some of these terms will cancel. However, it is not hard to see that 
    \begin{quote}
        $(\diamondsuit)$ for each $w \in \symm_n$, the monomial $x^{n-1}_{w(1)} x^{n-2}_{w(2)} \cdots x^1_{w(n-1)} x^0_{w(n)}$ appears precisely once in the expansion of \eqref{eq:containment-eight}, with no cancellation.
    \end{quote}
    The expansion of the product \eqref{eq:containment-six} consists precisely of those terms in the expansion of \eqref{eq:containment-eight} for which $x_i$ (instead of $- x_{t(i)}$) is selected in the factor $x_i - x_{t(i)}$ for each $1 \leq p \leq r$ and $i \in \tau_p$ with $i \neq i_p$. Applying $(\diamondsuit)$, we conclude 
    \begin{quote}
        $(\clubsuit)$ for each $w \in \symm_n$, the monomial $x^{n-1}_{w(1)} x^{n-2}_{w(2)} \cdots x^1_{w(n-1)} x^0_{w(n)}$ appears at most once in the expansion of \eqref{eq:containment-six}, with no cancellation, and this monomial has the same sign as in the expansion of \eqref{eq:containment-eight}.
    \end{quote}
    In light of $(\heartsuit), (\diamondsuit),$ and $(\clubsuit)$, the desired nonvanishing property \eqref{eq:containment-seven} will hold if we can show that
    \begin{quote}
        $(\spadesuit)$ there exists some $w \in \symm_n$ such that the monomial $x^{n-1}_{w(1)} x^{n-2}_{w(2)} \cdots x^1_{w(n-1)} x^0_{w(n)}$ appears in the expansion of \eqref{eq:containment-six}.
    \end{quote}

    We prove $(\spadesuit)$ with the following combinatorial argument. The $2^{n \choose 2}$ terms in the monomial expansion of \eqref{eq:containment-eight} correspond to orientations $\OOO$ of the complete graph $K_n$ on the vertex set $[n]$. Here we identify an oriented edge $i \to j$ in $\OOO$ with choosing the second term in the factor $\pm (x_i - x_j)$ of \eqref{eq:containment-eight}. The terms in the monomial expansion of \eqref{eq:containment-six} correspond
    to orientations $\OOO'$ of $K_n$ in which  each tree $\tau_p$ of the forest $F$ is oriented away from its distinguished vertex $i_p$. To build an orientation $\OOO'$ which gives rise to a term as in $(\spadesuit)$, start with the partial orientation on $K_n$ given by orienting each tree $\tau_p$ of $F$ away from its distinguished vertex $i_p$. Choose a leaf $j_1$ of $F$ arbitrarily, and orient every non-oriented
    edge involving $j_1$ towards $j_1$. There are $n-1$ edges oriented towards $j_1$, so the factor $x_{j_1}^{n-1}$ will appear in the corresponding term.
    Ignoring $j_1$ and its incident edges, choose another leaf $j_2$ of $F$ arbitrarily, and orient every non-oriented edge involving $j_2$ towards $j_2$. There are   
    $n-2$ edges oriented towards $j_2$, so the  factor $x_{j_2}^{n-2}$ will appear in the corresponding term. Iterating, we obtain an orientation $\OOO'$
    which yields the term  $\pm x_{j_1}^{n-1} x_{j_2}^{n-2} \cdots x_{j_n}^0$, which corresponds to the  permutation $w \in \symm_n$ with one-line notation
    $w = [j_1, j_2, \dots, j_n]$. This proves $(\spadesuit)$ and Lemma~\ref{lem:cospan_containment}.
\end{proof}

\end{document}